\documentclass{amsart}
\usepackage{amsmath,amssymb,ifthen}
\usepackage[all]{xy}
\usepackage{amsthm}
\usepackage{mathrsfs}
\usepackage{url}

\usepackage{yhmath}
\usepackage{mathtools}

\newcommand{\rs}{\mathrm{rs}}
\newcommand{\trs}{\theta\mathchar`-\mathrm{rs}}
\newcommand{\srs}{\mathrm{srs}}
\newcommand{\strs}{\theta\mathchar`-\mathrm{srs}}

\newcommand{\supp}{\mathrm{supp}}

\newcommand{\RHS}{\mathrm{RHS}}
\newcommand{\lan}{\langle}
\newcommand{\ran}{\rangle}
\newcommand{\ol}{\overline}

\newcommand{\mcO}{\mathcal{O}}
\newcommand{\mfp}{\mathfrak{p}}

\newcommand{\Z}{\mathbb{Z}}
\newcommand{\F}{\mathbb{F}}
\newcommand{\R}{\mathbb{R}}
\newcommand{\Q}{\mathbb{Q}}
\newcommand{\C}{\mathbb{C}}

\newcommand{\ra}{\rightarrow}
\newcommand{\Ra}{\Rightarrow}

\DeclareMathOperator{\tr}{tr}

\DeclareMathOperator{\GL}{GL}
\DeclareMathOperator{\SO}{SO}
\DeclareMathOperator{\SL}{SL}
\DeclareMathOperator{\Sp}{Sp}

\DeclareMathOperator{\cInd}{c-Ind}
\DeclareMathOperator{\Ind}{Ind}
\DeclareMathOperator{\Hom}{Hom}

\DeclareMathOperator{\Ad}{Ad}
\DeclareMathOperator{\Kl}{Kl}
\DeclareMathOperator{\Cent}{Cent}
\DeclareMathOperator{\diag}{diag}

\theoremstyle{plain}
\newtheorem{thm}{Theorem}[section]
\newtheorem*{thm*}{Theorem}
\newtheorem{prop}[thm]{Proposition}
\newtheorem{lem}[thm]{Lemma}
\newtheorem{cor}[thm]{Corollary}

\theoremstyle{definition}
\newtheorem{defn}[thm]{Definition}

\theoremstyle{remark}
\newtheorem{rem}[thm]{Remark}
\newtheorem*{claim*}{Claim}

\SelectTips{cm}{11}

\title{Endoscopic lifting of simple supercuspidal representations of $\SO_{2n+1}$ to $\GL_{2n}$}
\author{Masao Oi}
\address{Graduate School of Mathematical Sciences, 
the University of Tokyo, 3-8-1 Komaba, Meguro-ku, Tokyo 153-8914, Japan.}
\email{masaooi@ms.u-tokyo.ac.jp}

\begin{document}

\begin{abstract}
We compute the characters of simple supercuspidal representations of the twisted $\GL_{2n}(F)$ and $\SO_{2n+1}(F)$ for a $p$-adic field $F$.
Comparing them by the endoscopic character relation, we determine the liftings of simple supercuspidal representations of $\SO_{2n+1}(F)$ to $\GL_{2n}(F)$, 
under the assumption that $p$ is not equal to $2$.
\end{abstract}

\subjclass[2010]{Primary: 22E50; Secondary: 11F70, 11L05}

\maketitle

\section{Introduction}
Let $\mathbf{G}$ be a connected reductive group over a $p$-adic field $F$, and $G:=\mathbf{G}(F)$ its $F$-rational points.
Let $\Pi(G)$ be the set of equivalence classes of irreducible smooth representations of $G$,
and $\Phi(G)$ the set of $\widehat{G}$-conjugacy classes of $L$-parameters of $G$.
Here $\widehat{G}$ is the Langlands dual group of $\mathbf{G}$ and an $L$-parameter of $G$ is a homomorphism from the Weil-Deligne group $W_F\times \SL_2(\C)$ to the $L$-group ${}^L\mathbf{G}=\widehat{G}\rtimes W_F$ of $\mathbf{G}$.
Then the conjectural local Langlands correspondence predicts that there exists a natural map from $\Pi(G)$ to $\Phi(G)$ with finite fibers ($L$-packets).

For $\mathbf{G}=\GL_n$, it was established by Harris--Taylor \cite{MR1876802} and Henniart \cite{MR1738446}.
In this case, the map from $\Pi(G)$ to $\Phi(G)$ is bijective.

For quasi-split symplectic and orthogonal groups, the correspondence was established recently by Arthur in his book \cite{MR3135650}, under the assumption of the stabilization of twisted trace formulas for general linear groups and even orthogonal groups.
He characterized the $L$-packets for $L$-parameters of those groups via the $\mathit{endoscopic}$ $\mathit{character}$ $\mathit{relation}$, and constructed them in consequence of the comparison of stable trace formulas.

We explain the endoscopic character relation.
For simplicity, we consider the special orthogonal group $\SO_{2n+1}$, which is a main object in this paper.
This group is a twisted endoscopic group for $\GL_{2n}$, and there exists a natural $L$-embedding 
\[
\iota \colon {}^{L}\!\SO_{2n+1}=\Sp_{2n}(\C)\times W_F \hookrightarrow \GL_{2n}(\C)\times W_F={}^{L}\!\GL_{2n}.
\] 
Hence we can regard an $L$-parameter $\phi'$ of $\SO_{2n+1}(F)$ as an $L$-parameter $\phi$ of $\GL_{2n}(F)$ by composing it with $\iota$.
By using the local Langlands correspondence for $\GL_{2n}(F)$, we can get the representation $\pi$ of $\GL_{2n}(F)$ corresponding to $\phi$.
\[
\xymatrix{
\Pi\bigl(\GL_{2n}(F)\bigr) \ni \pi& \ar@{<~>}[r]^-{\text{LLC for $\GL_{2n}(F)$}} && W_F\times\SL_2(\C) \ar[rd]_-{\phi'} \ar[r]^-{\phi} & {}^{L}\!\GL_{2n}\\
\Pi\bigl(\SO_{2n+1}(F)\bigr) \supseteq \Pi_{\phi'}& \ar@{<~>}[r]^-{\text{LLC for $\SO_{2n+1}(F)$}} &&& {}^{L}\!\SO_{2n+1} \ar@{^{(}->}[u]_-{\iota} \\
}
\]
In this situation, $\pi$ is called the $\mathit{lifting}$ of the $L$-packet $\Pi_{\phi'}$ of $\phi'$.
When the $L$-parameter $\phi'$ is tempered, $\pi$ and $\Pi_{\phi'}$ satisfy the endoscopic character relation, which is an equality between the characters of representations in $\Pi_{\phi'}$ and the twisted character of $\pi$, and the correspondence $\pi\leftrightarrow\Pi_{\phi'}$ is characterized by this relation.

Therefore the local Langlands correspondence for $\SO_{2n+1}(F)$ is reduced to determining the liftings of representations of $\SO_{2n+1}(F)$ to $\GL_{2n}(F)$, and it is important to compute the characters of representations for this.

In this paper, we consider this problem for so-called $\mathit{simple}$ $\mathit{supercuspidal}$ representations of $\SO_{2n+1}(F)$.
These representations are supercuspidal representations which were introduced in \cite{MR2730575} and \cite{MR3164986}, and have been studied in the context of finding an explicit description of the local Langlands correspondence.

To state our main result, we explain some notations. 
From now on, we assume that the residual characteristic of $F$ is not equal to $2$.
Roughly speaking, simple supercuspidal representations are obtained by the compact induction of $\mathit{affine}$ $\mathit{generic}$ characters, 
and we can parametrize equivalence classes of such characters explicitly (see Section \ref{sec:ssc} for the details).
Then, for $\SO_{2n+1}(F)$, equivalence classes of simple supercuspidal representations of $\SO_{2n+1}(F)$ are parametrized by the finite set $k^{\times}\times\{\pm1\}$, where $k^{\times}$ is the multiplicative group of the residue field of $F$.
In a similar way to $\SO_{2n+1}(F)$, equivalence classes of self-dual simple supercuspidal representations of $\GL_{2n}(F)$ with trivial central character are parametrized by $k^{\times}\times\{\pm1\}$.
Under these parametrization, for $(a,\zeta)\in k^{\times}\times\{\pm1\}$ (resp.\ $(b,\xi)\in k^{\times}\times\{\pm1\}$), we denote by $\pi_{a, \zeta}$ (resp.\ $\pi'_{b, \xi}$) the corresponding simple supercuspidal representations of $\GL_{2n}(F)$ (resp.\ $\SO_{2n+1}(F)$). 
Then our main theorem on the liftings of simple supercuspidal representations of $\SO_{2n+1}(F)$ to $\GL_{2n}(F)$ is stated as follows:

\begin{thm}[Theorem \ref{thm:main}]\label{thm:intro}
We assume $p\neq2$. 
Let $b \in k^{\times}$ and $\xi \in \{\pm1\}$.
\begin{enumerate}
\item
The $L$-packet containing the simple supercuspidal representation $\pi'_{b, \xi}$ of $\SO_{2n+1}(F)$ is a singleton.
In particular, the character of $\pi'_{b, \xi}$ is stable.
\item
The lifting of the simple supercuspidal representation $\pi'_{b, \xi}$ of $\SO_{2n+1}(F)$ to $\GL_{2n}(F)$ is again simple supercuspidal, and given by $\pi_{2b, \xi}$.
\end{enumerate}
\end{thm}

We remark that the $L$-parameters of simple supercuspidal representations of general linear groups have been described explicitly by the works of \cite{MR2148193} and \cite{Imai:2015aa}.
From Theorem \ref{thm:intro}, we know that the $L$-parameter of $\pi'_{b, \xi}$ is equal to that of $\pi_{2b, \xi}$.
Therefore we get an explicit description of the $L$-parameters of simple supercuspidal representations of $\SO_{2n+1}(F)$. 

We explain the outline of our proof.
To prove Theorem \ref{thm:intro}, we consider the converse direction.
That is, we descend simple supercuspidal representations of $\GL_{2n}(F)$ instead of lifting simple supercuspidal representations of $\SO_{2n+1}(F)$.
We first take the self-dual simple supercuspidal representation $\pi_{2b, \xi}$ of $\GL_{2n}(F)$ for $b\in k^{\times}$ and $\xi \in \{\pm1\}$, and show that the $L$-parameter $\phi$ of $\pi_{2b,\xi}$ factors through an $L$-parameter $\phi'$ of $\SO_{2n+1}(F)$, and that its $L$-packet $\Pi_{\phi'}$ is a singleton.
Then we show that the unique representation in $\Pi_{\phi'}$ is $\pi'_{b, \xi}$.

The first step follows from general results.
Since the representation $\pi_{2b,\xi}$ is self-dual, the $L$-parameter $\phi$ of $\pi_{2b,\xi}$ is also self-dual by a property of the local Langlands correspondence for $\GL_{2n}(F)$.
Hence the image of $\phi$ in $\widehat{\GL_{2n}}=\GL_{2n}(\C)$ is contained in either $\mathrm{O}_{2n}(\C)$ or $\Sp_{2n}(\C)$.
By a result in \cite{Mieda:2016}, the image is in $\Sp_{2n}(\C)=\widehat{\SO_{2n+1}}$, 
hence $\phi$ factors through an $L$-parameter $\phi'$ of $\SO_{2n+1}(F)$.
Then, by the local classification theorem in \cite{MR3135650} and the theorem for parametrizing supercuspidal representations of $\SO_{2n+1}(F)$ in \cite{MR3713922}, we know that the $L$-packet for $\phi'$ is a singleton consisting of a supercuspidal representation $\pi'$.

The key point of the proof is the second step to show that this representation $\pi'$ is in fact simple supercuspidal.
Our strategy is to compare characters of $\pi_{2b,\xi}$ and $\pi'$.
We first compute the twisted character of the simple supercuspidal representation $\pi_{2b,\xi}$ at special elements, which are called $\mathit{affine}$ $\mathit{generic}$ elements.
By using the twisted character formula for supercuspidal representations, we write these character values explicitly in terms of Kloosterman sums.
We remark that such a computation was already done by many people in the case of the standard $\GL_n$ (for example, \cite{MR3843393}), but our computation in this paper is more conceptual and valid for other groups.

Then, by the endoscopic character relation, we can express the character of $\pi'$ in terms of the twisted character of $\pi_{2b,\xi}$, which is already computed.
From this relation, we can show that $\pi'$ is either simple supercuspidal or depth-zero supercuspidal.
To eliminate the possibility that $\pi'$ is depth-zero supercuspidal, we next compute the character of depth-zero supercuspidal representations, and compare them.
Once we know that $\pi'$ is simple supercuspidal, we can show that $\pi'=\pi'_{b, \xi}$ easily by computing the characters of simple supercuspidal representations of $\SO_{2n+1}(F)$ and considering the Fourier transform of Kloosterman sums, and this completes the proof.

We remark that the liftings of simple supercuspidal representations of $\SO_{2n+1}(F)$ to $\GL_{2n}(F)$ were determined in \cite{MR3518182} under the assumption that $p\geq(2+e)(2n+1)$, where $e$ is the ramification index of $F$ over $\Q_p$.
Hence our results are new for odd primes less than $(2+e)(2n+1)$.
In his proof, he uses the stability of the characters of simple supercuspidal representations of $\SO_{2n+1}(F)$ in \cite{MR3402796}, M{\oe}glin's result about the stability of $L$-packets in \cite{MR3220932}, and special arguments for discrete series representations of $\GL_{2n}(F)$.
On the other hand, in our proof, the stability is naturally deduced from Arthur's theorem.

We also remark that our method based on the endoscopic character relation is basically valid for other classical groups.
For example, a quasi-split unitary group $\mathrm{U}_{E/F}(N)$ is a twisted endoscopic group for $\mathrm{Res}_{E/F}(\GL_N)$, where $E/F$ is a quadratic extension of $p$-adic fields.
The endoscopic classification of representations for these groups has also been established in \cite{MR3338302}, and we can apply the same argument for them.
We proved the same type result when $E$ is unramified over $F$ in \cite{Oi:2016aa}, and the other cases are in progress now.

Finally, we explain the organization of this paper.
In Section \ref{sec:ssc}, we review some fundamental properties about Iwahori subgroups and simple supercuspidal representations.
In addition, we introduce the notion of affine genericity for elements in Iwahori subgroups, which will play important roles in a comparison of characters.
In Section \ref{sec:char}, we compute the characters of simple supercuspidal representations of the twisted $\GL_{2n}(F)$ and $\SO_{2n+1}(F)$ at affine generic elements.
In Section \ref{sec:norm}, we investigate the norm correspondence for $\GL_{2n}$ and $\SO_{2n+1}$.
The norm correspondence is used to formulate the endoscopic character relation.
We determine norms of some special affine generic elements and compute their transfer factors.
In Section \ref{sec:main}, we first recall the endoscopic character relation in \cite{MR3135650}.
Then we determine the liftings of simple supercuspidal representations by combining the endoscopic character relation with the results in Sections \ref{sec:char} and \ref{sec:norm}.
In Appendix \ref{sec:Kl}, we list some properties about Gauss and Kloosterman sums.

\medbreak
\noindent{\bfseries Acknowledgment.}\quad
This paper is part of the master's thesis of the author.
He would like to thank his advisor Yoichi Mieda for his constant support and encouragement.
He always corrected the author's misunderstanding and led him to the right direction.
The author would not have been able to write this paper without his guidance.
He is grateful to Moshe Adrian, Naoki Imai, Tasho Kaletha, Teruhisa Koshikawa and Koji Shimizu for many helpful comments and pointing out a lot of mistakes and typos in the draft version of this paper.
He would like to thank Takahiro Tsushima for teaching him some techniques concerning the Fourier transform of Kloosterman sums.
Finally, He wishes to express his sincere gratitude to the referee for carefully reading this paper and giving him many constructive comments.

This work was carried out with the support from the Program for Leading Graduate Schools, MEXT, Japan.
This work was also supported by JSPS Research Fellowship for Young Scientists and KAKENHI Grant Number 17J05451.

\setcounter{tocdepth}{2}
\tableofcontents

\medbreak
\noindent{\bfseries Notation.}\quad
Let $p$ be an odd prime number.
We fix a $p$-adic field $F$.
We denote its ring of integers, its maximal ideal, and its residue field by $\mcO$, $\mfp$, and $k$, respectively.
We fix a uniformizer $\varpi$ of $F$.
Let $q$ be the order of $k$.
For $x \in \mcO$, $\bar{x}$ denotes the image of $x$ in $k$.

For an algebraic group $\mathbf{G}$ over $F$, we denote its $F$-rational points $\mathbf{G}(F)$ by $G$. 
For an algebraic group $\mathbf{T}$, we write $X^\ast(\mathbf{T})$ for its character group and $X_{\ast}(\mathbf{T})$ for its cocharacter group.

Throughout this paper, we fix an additive character $\psi$ on $F$ of level one.
Then its restriction $\psi|_{\mcO}$ to $\mcO$ induces a nontrivial additive character on $k$.
We denote it by $\psi$ again.

\section{Simple supercuspidal representations}\label{sec:ssc}
\subsection{Iwahori subgroups}
Let $\mathbf{G}$ be a connected split reductive group over $F$, and $\mathbf{Z}$ its center.
Let $\mathbf{T}$ be an $F$-split maximal torus in $\mathbf{G}$.
We denote the set of roots of $\mathbf{T}$ in $\mathbf{G}$ by $\Phi$, and the set of affine roots by $\Psi$.
For each root $a \in \Phi$, we denote by $\mathbf{U}_{a}$ the corresponding root subgroup of $\mathbf{G}$.
For each affine root $\alpha \in \Psi$, we denote by $U_\alpha$ the corresponding affine root subgroup of $G=\mathbf{G}(F)$.

We fix an alcove $C$ in the apartment $\mathcal{A}(\mathbf{G}, \mathbf{T})\cong X_{\ast}(\mathbf{T})\otimes_\Z\R$ of $\mathbf{T}$ in $\mathbf{G}$.
This determines an affine root basis $\Pi$ of $\Psi$ and the set $\Psi^+$ of positive affine roots.
We set the Iwahori subgroup associated to $C$ and its subgroups as follows: 
\begin{align*}
I &:= \lan T_0, U_\alpha \mid \alpha \in \Psi^+\ran,\\
I^+ &:= \lan T_1, U_\alpha \mid \alpha \in \Psi^+\ran \text{, and}\\
I^{++} &:= \lan T_1, U_\alpha \mid \alpha \in \Psi^+ \setminus \Pi\ran,
\end{align*}
where $T_0$ is the maximal compact subgroup of $T=\mathbf{T}(F)$, and 
\[
T_1 := \{t\in T_0 \mid \lambda(t)\in 1+\mfp \text{ for every } \lambda\in X^\ast(\mathbf{T})\}.
\]
These groups are the first three steps of the Moy-Prasad filtration of the Iwahori subgroup $I$ associated to the barycenter of the alcove $C$ (see \cite[Section 2.6]{MR3164986}).

We define subgroups of $T$ and $Z=\mathbf{Z}(F)$ by
\begin{align*}
T(q) &:= \{t\in T \mid t^q=t\} \text{ and } \\
Z(q) &:= \{t\in Z \mid t^q=t\},
\end{align*}
respectively.
These are sets of representatives of $T_0/T_1$ and $(Z\cap T_0)/(Z\cap T_1)$.

\begin{prop}[{\cite[Lemma 9.2]{MR2730575}}]\label{prop:I}
\begin{enumerate}
 \item The subgroup $I^+$ is normal in $I$, and we have
 \[
 I/I^+ \cong T_0/T_1 \cong T(q).
 \]
 \item The subgroup $I^{++}$ is normal in $I^+$, and we have
 \[
 I^+/I^{++} \cong \bigoplus_{\alpha \in \Pi } U_\alpha/U_{\alpha+1}.
 \]
\end{enumerate}
\end{prop}

We denote the image of $x$ under the map $I^+ \twoheadrightarrow \bigoplus_{\alpha \in \Pi } U_\alpha/U_{\alpha+1}$ by $(x_\alpha)_{\alpha \in \Pi}$, and call each $x_\alpha$ the $\mathit{simple}$ $\mathit{affine}$ $\mathit{component}$ of $x$.

\begin{defn}
\begin{enumerate}
 \item An element $x \in I^+$ is said to be $\mathit{affine}$ $\mathit{generic}$ if $x_\alpha$ is nonzero for every $\alpha \in \Pi$.
 \item A character $\chi \colon ZI^+ \ra \C^{\times}$ is called $\mathit{affine}$ $\mathit{generic}$ if $\chi|_{I^{+}}$ factors through the quotient $I^+/I^{++}$ and is nontrivial on $U_\alpha/U_{\alpha+1}$ for every $\alpha \in \Pi$.
\end{enumerate}
\end{defn}

Let $\mathbf{N}_{\mathbf{T}}$ be the normalizer of $\mathbf{T}$ in $\mathbf{G}$, and $\widetilde{W}$ the Iwahori-Weyl group of $\mathbf{T}$ defined by 
\[
\widetilde{W} := N_{T}/T_{0} = \mathbf{N}_{\mathbf{T}}(F)/T_{0}.
\]
Then we have the following proposition.

\begin{prop}[{\cite[Proposition 8 and Lemma 14]{MR2435422}}]\label{prop:IW}
\begin{enumerate}
 \item We have $G=IN_{T}I$, and the map $InI \mapsto \dot{n}$ induces a bijection $I\backslash G/I \cong \widetilde{W}$.
 \item
 There exists an exact sequence 
 \[
 1 \ra W_\mathrm{aff} \ra \widetilde{W} \xrightarrow{\kappa_G} X^{\ast}\bigl(Z(\widehat{G})\bigr) \ra 1,
 \]
 where $W_\mathrm{aff}$ is the affine Weyl group of $\mathbf{T}$, $Z(\widehat{G})$ the center of the Langlands dual group $\widehat{G}$ of $\mathbf{G}$, and $\kappa_G$ the Kottwitz homomorphism defined in \cite[Section 7]{MR1485921}.
Moreover the subgroup $\widetilde{\Omega} \subseteq \widetilde{W}$ consisting of the elements normalizing $I$ maps isomorphically to $X^{\ast}(Z(\widehat{G}))$, and we have $\widetilde{W} \cong W_\mathrm{aff} \rtimes \widetilde{\Omega}$.
\end{enumerate}
\end{prop}

By the same argument in the proof in \cite{MR2435422}, we can easily generalize this proposition as follows:

\begin{prop}\label{prop:I^+W}
We have $G=I^{+}N_{T}I^{+}$, and the map $I^+nI^+ \mapsto \dot{n}$ induces a bijection 
\[
I^+ \backslash G / I^+ \cong N_{T}/T_1.
\]
\end{prop}

We fix a set of representatives $\Omega \subseteq N_{T}$ of $\widetilde{\Omega} \subseteq \widetilde{W}=N_{T}/T_0$.
Let $N_G(I)$ and $N_G(I^+)$ be the normalizers of $I$ and $I^+$ in $G$, respectively.

\begin{lem}\label{lem:NI}
We have
\[
N_G(I)=N_G(I^+)=I\Omega.
\]
\end{lem}

\begin{proof}
Since $I^+$ is the pro-unipotent radical of $I$, we have $N_G(I) \subseteq N_G(I^+)$.
We prove the other inclusion.
Let $g \in N_G(I^+)$.
It suffices to prove $I^+gI^+ \subseteq N_G(I)$.
By Proposition \ref{prop:I^+W}, we can replace $g$ with $n \in N_{T}$.
As $nI^+n^{-1}=I^+$ and $nT_{0}n^{-1}=T_{0}$, we have $n \in N_G(I)$. 
Hence $N_G(I)=N_G(I^+)$.

We next prove the second equality.
The inclusion $N_G(I) \supseteq I\Omega$ follows from the definition of $\Omega$.
Let $g \in N_G(I)$.
It suffices to prove $IgI \subseteq I\Omega$.
By Proposition \ref{prop:IW} (1), we may assume $g \in N_{T}$.
Then we have $g \in \Omega T_0$ by the definition of $\Omega$.
Hence $IgI\subseteq I\Omega$.
\end{proof}

The following lemma is a key to compute characters.

\begin{lem}\label{lem:key}
Let $y \in G$.
If $y$ satisfies $ygy^{-1} \in I$ for an affine generic element $g \in I^+$, then $y \in I\Omega$.
\end{lem}

\begin{proof}
Let $y \in G$ satisfying $ygy^{-1} \in I$ for an affine generic element $g \in I^+$.
Since affine genericity is preserved by $I^+$-conjugation, any element of $IyI^+$ satisfies the same condition as $y$.
Therefore, by Proposition \ref{prop:I^+W}, we may assume $y \in N_{T}$.

As $yIy^{-1}$ and $I$ have the same volume for a Haar measure of $G$, we only have to show $yIy^{-1}\subseteq I$.

We recall that the multiplication map 
\[
\prod_{a\in\Phi^+}\mathbf{U}_{a}(F)\times T \times\prod_{a\in\Phi^-}\mathbf{U}_{a}(F) \ra G
\]
is injective in any order (see \cite[2.2.3]{MR756316}), where $\Phi^{+}$ (resp.\ $\Phi^{-}$) is the set of positive (resp.\ negative) roots.
Moreover this induces a bijection 
\[
\prod_{a\in\Phi^+}U_{a+r_a}\times T_0\times\prod_{a\in\Phi^-}U_{a+r_a} \ra I
\]
in any order, where $r_a\in\Z$ is the smallest integer such that $a+r_a\in\Psi^+$ (see \cite[6.4]{MR0327923}).
By using this decomposition, we write 
\[
g=\prod_{a\in\Phi^+}x_a\cdot t \cdot \prod_{a\in\Phi^-}x_a,
\]
where $t\in T_0$ and $x_a\in U_{a+r_a}$ for each $a\in\Phi$.
Then, since $g$ is affine generic, for a simple affine root $\alpha=a+r_{a}\in\Pi$,
 $x_{a}$ belongs to $U_{a+r_{a}}$ but not to $U_{a+r_{a}+1}$.
Therefore a simple affine root $\alpha\in\Pi$ is the maximal affine root such that
\begin{itemize}
\item its gradient is $a$ and
\item $x_{a}$ belongs to $U_{\alpha}$.
\end{itemize}
Hence the affine root $y\alpha y^{-1}\in\Psi$ is the maximal affine root such that
\begin{itemize}
\item its gradient is $yay^{-1}$ and
\item $yx_{a}y^{-1}$ belongs to $U_{y\alpha y^{-1}}$.
\end{itemize}

On the other hand, we have
\[
ygy^{-1}=\prod_{a\in\Phi^+}yx_ay^{-1}\cdot yty^{-1} \cdot \prod_{a\in\Phi^-}yx_ay^{-1}.
\]
By the above uniqueness of expression, the assumption $ygy^{-1}\in I$ implies that $yx_ay^{-1}$ belongs to $U_{yay^{-1}+r_{yay^{-1}}}\subseteq\mathbf{U}_{yay^{-1}}(F)$ for every $a\in\Phi$.

Therefore, for every simple affine root $\alpha=a+r_{a}\in\Pi$, by the maximality of $y\alpha y^{-1}$, we get
\[
y\alpha y^{-1}\geq yay^{-1}+r_{yay^{-1}}\geq0.
\]
In particular, $y\alpha y^{-1}$ is a positive affine root for every simple affine root $\alpha\in\Pi$.
Hence $y\alpha y^{-1}$ is positive for every positive affine root $\alpha\in\Psi^{+}$.
This implies $yIy^{-1}\subseteq I$ (note that we have $yT_0y^{-1} =T_0$ since $y \in N_{T}$).
\end{proof}

\subsection{Simple supercuspidal representations}
Let $\chi$ be an affine generic character on $ZI^{+}$.
For $g\in G$ and a subgroup $J$ of $G$, we set $J^{g}:=g^{-1}Jg$.
We put
\[
 N_{G}(I^{+}; \chi) := \{n \in N_G(I^+) \mid \chi^n=\chi\},
\]
where $\chi^n$ is the character of $(ZI^+)^{n}=n^{-1}(ZI^{+})n=ZI^{+}$ defined by $\chi^n(g):=\chi(ngn^{-1})$.
This subgroup satisfies $ZI^+ \subseteq N_{G}(I^{+}; \chi) \subseteq I\Omega$ by Lemma \ref{lem:NI}.
For an irreducible constituent $\tilde\chi$ of $\cInd_{ZI^+}^{N_{G}(I^{+}; \chi)}\chi$, we define 
\[
\pi_{\tilde\chi} := \cInd_{N_{G}(I^{+}; \chi)}^{G} \tilde\chi.
\]

\begin{rem}\label{rem:ssc}
\begin{enumerate}
 \item If $\mathbf{G}$ is split, semisimple, and simply connected, then the group $\widetilde{\Omega}$ is trivial by Proposition \ref{prop:IW} (2).
 Hence we have $N_{G}(I^{+})=ZI$.
 By using $I=T(q)I^{+}$, we can check easily
 \[
 \{n \in ZI \mid \chi^n=\chi\}=ZI^{+}.
 \]
 Thus we get $N_{G}(I^{+}; \chi)=ZI^{+}$.
 \item In this paper, we will consider the cases of $\mathbf{G}=\GL_{N}$ and $\mathbf{G}=\SO_{2n+1}$.
 For these groups, the quotient $N_{G}(I^{+}; \chi)\twoheadrightarrow N_{G}(I^{+}; \chi)/ZI^+$ splits, and $\cInd_{ZI^+}^{N_{G}(I^{+}; \chi)}\chi$ decomposes as a direct sum of characters.
\end{enumerate}
\end{rem}

We next consider the decomposition of the compact induction of $\chi$ to $G$.
The following proposition is proved in \cite[Proposition 9.3]{MR2730575} for $\mathbf{G}$ which is split simple simply connected, and in {\cite[Proposition 2.4]{MR3164986}} for $\mathbf{G}$ which is semisimple and tamely ramified.
We explain the proof for the sake of completeness.

\begin{prop}\label{prop:ssc}
\begin{enumerate}
 \item We have a decomposition 
 \[
 \cInd_{ZI^+}^{G}\chi \cong \bigoplus_{\tilde\chi} \dim(\tilde\chi)\cdot\pi_{\tilde\chi},
 \]
 where the sum is over the set of irreducible constituents of $\cInd_{ZI^+}^{N_{G}(I^{+}; \chi)}\chi$.
 \item The representation $\pi_{\tilde\chi}$ is irreducible, hence supercuspidal.
 \item Let $(\chi', {\tilde\chi}')$ be another pair as above. Then, $\pi_{\tilde\chi}$ and $\pi_{\tilde\chi'}$ are equivalent if and only if $\chi^n=\chi'$ and ${\tilde\chi}^n\cong{\tilde\chi}'$ for some $n \in T_{0}\Omega$.
\end{enumerate}
\end{prop}

\begin{proof}
Since $ZI^{+}$ is normal in $N_{G}(I^{+}; \chi)$ and the quotient $N_{G}(I^{+}; \chi)/ZI^{+}$ is finite, we have
\[
\cInd_{ZI^+}^{N_{G}(I^{+}; \chi)}\chi \cong \bigoplus_{\tilde\chi} \dim(\tilde\chi)\cdot\tilde\chi.
\]
Hence it suffices to prove (2) and (3) by the transitivity of compact induction.
By Mackey's theorem, we have
\[
\Hom_{G}(\pi_{\tilde\chi}, \pi_{\tilde\chi'}) \cong \bigoplus_{n \in N_{G}(I^{+}; \chi)\backslash G/N_{G}(I^{+}; \chi')} \Hom_{N_{G}(I^{+}; \chi)^n \cap N_{G}(I^{+}; \chi')} ({\tilde\chi}^n, {\tilde\chi}').
\]
Let $n \in G$ such that $\Hom_{N_{G}(I^{+}; \chi)^n \cap N_{G}(I^{+}; \chi')} ({\tilde\chi}^n, {\tilde\chi}')\neq0$.
Then we may assume $n\in N_{T}$ by Proposition \ref{prop:I^+W}.
Since $(ZI^{+})^{n}\cap ZI^{+}\subseteq N_{G}(I^{+}; \chi)^{n}\cap N_{G}(I^{+}; \chi)$, we also have $\Hom_{(ZI^{+})^{n}\cap ZI^{+}} ({\tilde\chi}^n, {\tilde\chi}')\neq 0$.
As ${\tilde\chi}^{n}|_{(ZI^{+})^{n}}=({\chi}^{n})^{\oplus \dim\tilde{\chi}}$, and ${\tilde\chi}'|_{ZI^{+}}={\chi}'^{\oplus \dim\tilde{\chi}'}$, we have ${\chi}^{n}={\chi}'$ on $(ZI^{+})^{n}\cap ZI^{+}$.

We show that the image $w$ of $n$ under 
\[
I^+ \backslash G / I^+ \twoheadrightarrow W_\mathrm{aff} \rtimes\widetilde{\Omega}
\]
lies in $\widetilde{\Omega}$.
We assume that $w\notin\widetilde{\Omega}$.
Then we can take a simple affine root $\alpha\in\Pi$ such that $w(\alpha)\in\Psi^{+}\setminus\Pi$ (see \cite[Lemma 9.1]{MR2730575}).
By the definition of $I^{++}$, $U_{w(\alpha)}$ is contained in $I^{++}$.
Hence $\chi$ is trivial on $U_{w(\alpha)}$, and $\chi^{n}$ is trivial on $n^{-1}U_{w(\alpha)}n=U_{\alpha}$.
Since $U_{\alpha}=n^{-1}U_{w(\alpha)}n \subseteq (ZI^{+})^{n}\cap ZI^{+}$, we have $\chi^{n}=\chi'$ on $U_{\alpha}$.
However $\chi'$ is nontrivial on $U_{\alpha}$ since $\alpha\in\Pi$.
This is a contradiction.

As $w \in \widetilde{\Omega}$, $n$ normalizes $I$, hence also normalizes $I^+$.
Therefore we have $ZI^{+}\subseteq N_{G}(I^{+}; \chi)^{n}\cap N_{G}(I^{+}; \chi')$, and $\chi^{n}=\chi'$.
Since $N_{G}(I^{+}; \chi)^{n}=N_{G}(I^{+}; \chi^{n})=N_{G}(I^{+}; \chi')$, we have ${\tilde\chi}^{n}\cong{\tilde\chi}'$.

We finally show the irreducibility.
We take $\tilde\chi=\tilde\chi'$.
Then the above calculation shows that the only contribution to $\Hom_{G}(\pi_{\tilde\chi}, \pi_{\tilde\chi})$ comes from the trivial double coset of $N_{G}(I^{+}; \chi)\backslash G/N_{G}(I^{+}; \chi)$ and we have 
\[
\Hom_{G}(\pi_{\tilde\chi}, \pi_{\tilde\chi}) \cong \Hom_{N_{G}(I^{+}; \chi)}(\tilde\chi, \tilde\chi).
\]
The dimension of the right-hand side is one and therefore $\pi_{\tilde\chi}$ is irreducible.
\end{proof}

The irreducible supercuspidal representations $\pi_\chi$ constructed in this way are called $\mathit{simple}$ $\mathit{supercuspidal}$ representations of $G$.

\subsection{Parametrization: the case of $\GL_N(F)$}
In this subsection, we consider the case of $\GL_N$.
Let $\mathbf{G}$ be $\GL_N$ over $F$.
We choose $\mathbf{T}$ to be the subgroup of diagonal matrices.
Then we have 
\begin{align*}
\Phi&=\{\pm(e_i-e_j) \mid 1\leq i<j\leq N\}, \text{ and}\\
\Psi&=\{a+r \mid a \in \Phi, r \in \Z \}.
\end{align*}
We take the root basis 
\[
\Delta=\{e_1-e_2, \ldots, e_{N-1}-e_N\} 
\]
corresponding to the Borel subgroup $\mathbf{B}$ consisting of upper triangular matrices.
We let $C$ be the fundamental alcove of $\mathcal{A}(\mathbf{G}, \mathbf{T})$ (i.e., $C$ is contained in the chamber which is defined by $\mathbf{B}$, and the closure $\ol{C}$ of $C$ contains $0$).
Then the corresponding affine root basis is 
\[
\Pi=\{e_1-e_2, \ldots, e_{N-1}-e_N, e_N-e_1+1\}, 
\]
and the Iwahori subgroup and its filtrations are given by
\[
I = \begin{pmatrix}
 \mcO^{\times}&&\mcO\\
 &\ddots&\\
 \mfp&&\mcO^{\times}
\end{pmatrix}, \,
I^+ = \begin{pmatrix}
 1+\mfp&&\mcO\\
 &\ddots&\\
 \mfp&&1+\mfp
\end{pmatrix} \text{, and }
\]
\[
I^{++} = \begin{pmatrix}
 1+\mfp&\mfp&&\mcO\\
 &\ddots&\ddots&\\
 &\mfp&\ddots&\mfp\\
\mfp^2&&&1+\mfp
\end{pmatrix}.
\]
For $x=(x_{ij})_{ij} \in I^+$, we regard its simple affine components $(x_\alpha)_{\alpha} \in \bigoplus_{\Pi} U_\alpha/U_{\alpha+1}$ as an element of $k^{\oplus N}$ by
\begin{align*}
I^+/I^{++} &\cong \bigoplus_{\alpha\in\Pi}U_\alpha/U_{\alpha+1} \cong k^{\oplus N} \\
(x_{ij})_{ij} &\mapsto \left(\ol{x_{1 2}}, \ldots, \ol{x_{N-1, N}}, \ol{x_{N 1}\varpi^{-1}}\right).
\end{align*}

For $a \in k^{\times}$, we set 
\[
\varphi_a := 
\begin{pmatrix}
0 & I_{N-1} \\
\varpi a & 0 
\end{pmatrix} \in G.
\]
Here, we regard $a$ as an element of $F^{\times}$ by the Teichm$\ddot{\mathrm{u}}$ller lift.
This element satisfies $\varphi_a^N=\varpi aI_N$, and, for any $a\in k^{\times}$, we can choose a set of representatives $\Omega$ of $\widetilde{\Omega}$ to be $\lan\varphi_{a}\ran$.

For $a \in k^{\times}$, we define an affine generic character $\psi_{a} \colon ZI^+ \ra \C^{\times}$ by
\begin{align*}
\psi_{a}(z) &= 1 \text{ for $z \in Z$, and}\\
\psi_{a}(x)&:=\psi\left(\ol{x_{12}}+\cdots+\ol{x_{N-1, N}}+a\ol{x_{N1}\varpi^{-1}}\right) \text{ for $x=(x_{ij})_{ij} \in I^+$},
\end{align*}
where $\psi$ is the fixed additive character on $k$.
Then we have $N_{G}(I^{+}; \psi_a)=ZI^+\lan\varphi_{a^{-1}}\ran$.

For an $N$-th root of unity $\zeta \in \mu_N$, 
let $\chi_{a, \zeta} \colon ZI^+\lan\varphi_{a^{-1}}\ran \ra \C^{\times}$ be the character defined by
\begin{align*}
 \chi_{a, \zeta}(x) &= \psi_a(x) \text{ for $x \in ZI^+$, and}\\
 \chi_{a, \zeta}(\varphi_{a^{-1}}) &= \zeta. 
\end{align*}
Let $\pi_{a, \zeta}$ be the simple supercuspidal representation of $\GL_N(F)$ defined by
\[
\pi_{a, \zeta}:=\cInd^{G}_{ZI^+\lan\varphi_{a^{-1}}\ran} \chi_{a, \zeta}.
\]
Then, by Proposition \ref{prop:ssc} (3), we can check that the set 
\[
\{(a, \zeta) \mid a \in k^{\times}, \zeta \in \mu_N\}
\]
parametrizes the set of equivalence classes of simple supercuspidal representations of $\GL_N(F)$ with trivial central character.

\subsection{Parametrization: the case of $\SO_{2n+1}(F)$}
In this subsection, we consider the case of  
\[
\SO_{2n+1} := \bigl\{g \in \GL_{2n+1} \mid {}^{t}\!gJg=J, \det(g)=1 \bigr\}, 
\]
with
\[
J = \begin{pmatrix}
 &&&1\\
 &&-1&\\
 &\adots&&\\
 (-1)^{2n}&&&
\end{pmatrix}.
\] 
Let $\mathbf{H}$ be $\SO_{2n+1}$ over $F$.
Let $\mathbf{T}_{\mathbf{H}}$ be the subgroup of diagonal matrices in $\mathbf{H}$:
\[
\mathbf{T}_{\mathbf{H}}:=\bigl\{\diag(t_1, \ldots, t_n, 1, t_n^{-1}, \ldots, t_1^{-1}) \mid t_i\neq0\bigr\}.
\]
Then we have 
\begin{align*}
\Phi_{\mathbf{H}}&=\{\pm e_i\pm e_j \mid 1\leq i<j\leq n\} \cup \{ \pm e_i \mid 1\leq i \leq n\}, \text{ and}\\
\Psi_{\mathbf{H}}&=\{a+r \mid a \in \Phi_{\mathbf{H}}, r \in \Z \}.
\end{align*}

We take the root basis 
\[
\Delta_{\mathbf{H}}=\{e_1-e_2, \ldots, e_{n-1}-e_n, e_n\}
\]
corresponding to the Borel subgroup $\mathbf{B}_{\mathbf{H}}$ consisting of upper triangular matrices in $\mathbf{H}$.
We let $C_{\mathbf{H}}$ be the fundamental alcove of $\mathcal{A}(\mathbf{H}, \mathbf{T}_{\mathbf{H}})$.
Then the corresponding affine root basis is 
\[
\Pi_{\mathbf{H}}=\{e_1-e_2, \ldots, e_{n-1}-e_n, e_n, -e_1-e_2+1\}.
\]
We denote the Iwahori subgroup and its subgroups by $I_{H}$, $I_{H}^{+}$, and $I_{H}^{++}$.

For $y=(y_{ij})_{ij} \in I_H^+$, we regard its simple affine components $(y_\alpha)_{\alpha} \in \bigoplus_{\Pi_{\mathbf{H}}} U_\alpha/U_{\alpha+1}$ as an element of $k^{\oplus (n+1)}$ by
\begin{align*}
I_{H}^+/I_{H}^{++} &\cong \bigoplus_{\alpha\in\Pi_{\mathbf{H}}}U_\alpha/U_{\alpha+1} \cong k^{\oplus (n+1)} \\
(y_{ij})_{ij} &\mapsto \left(\ol{y_{1 2}}, \ldots, \ol{y_{n, n+1}}, \ol{y_{2n, 1}\varpi^{-1}}\right).
\end{align*}

For $b \in k^{\times}$, we set 
\[
\varphi'_b := 
-\begin{pmatrix}
 &&\varpi^{-1}b^{-1}\\
 &I_{2n-1}&\\
 \varpi b&& 
 \end{pmatrix} \in H.
\]
Here, we regard $b$ as an element of $F^{\times}$ by the Teichm$\ddot{\mathrm{u}}$ller lift.
The order of this element is two, and, for any $b\in k^{\times}$, we can choose a set of representatives $\Omega$ of $\widetilde{\Omega}$ to be $\lan\varphi'_{b^{-1}}\ran$.

For $b \in k^{\times}$, we define an affine generic character $\psi'_b \colon I_H^+ \ra \C^{\times}$ by 
\[
\psi'_b(y) := \psi\left(\ol{y_{12}}+\cdots+\ol{y_{n, n+1}}+b\ol{y_{2n, 1}\varpi^{-1}}\right) \text{ for $y=(y_{ij})_{ij} \in I_H^+$}.
\]
Then we have $N_{H}(I_{H}^{+}; \psi'_b)=I_H^+\lan\varphi'_{b^{-1}}\ran$.

For $\xi \in \{\pm1\}$, 
let $\chi'_{b, \xi} \colon I_H^+\lan\varphi'_{b^{-1}}\ran \ra \C^{\times}$ be the character defined by
\begin{align*}
 \chi'_{b, \xi}(y) &= \psi'_b(y) \text{ for $y \in I_H^+$ and }\\
 \chi'_{b, \xi}(\varphi'_{b^{-1}}) &= \xi. 
\end{align*}
Let $\pi'_{b, \xi}$ be the simple supercuspidal representation of $\SO_{2n+1}(F)$ defined by
\[
\pi'_{b, \xi}:=\cInd^{H}_{I_H^+\lan\varphi'_{b^{-1}}\ran} \chi'_{b, \xi}.
\]
Then, by Proposition \ref{prop:ssc} (3), we can check that the set 
\[
\{(b, \xi) \mid b \in k^{\times}, \xi \in \{\pm1\}\}
\]
parametrizes the set of equivalence classes of simple supercuspidal representations of $\SO_{2n+1}(F)$.

\section{Characters of simple supercuspidal representations}\label{sec:char}
\subsection{The case of the standard $\GL_N(F)$}
Let us first recall the characters of representations of $p$-adic reductive groups.
For a connected reductive group $\mathbf{G}$ over $F$, 
we write $G^{\rs}$ for the set of regular semisimple elements of $G$.
This is an open subset of $G$.
We denote by $\mathcal{H}(G)$ the set of compactly supported locally constant functions on $G$.

\begin{thm}[\cite{MR0414797}]\label{thm:char}
Let $\mathbf{G}$ be a connected reductive group over $F$.
Let $\pi$ be an irreducible smooth representation of $G$.
Then there exists a unique locally constant function $\Theta_\pi$ on $G^{\rs}$ such that 
\[
\tr \pi(f) = \int_{G^{\rs}} \Theta_\pi(g) f(g)\,dg
\]
for every $f \in \mathcal{H}(G)$ satisfying $\supp(f) \subseteq G^{\rs}$, 
where $\tr \pi$ is the distribution character of $\pi$.
\end{thm}

We call $\Theta_\pi$ the $\mathit{character}$ of $\pi$.
This function is invariant under conjugation.

When $\pi$ is a supercuspidal representation which is compactly induced from an open subgroup, 
we have the following formula to describe its character.

\begin{thm}[Character formula, \cite{MR1039842}]\label{thm:CF}
Let $\mathbf{G}$ be a connected reductive group over $F$ and $\mathbf{Z}$ its center.
Let $K$ be an open subgroup of $G$ such that $K$ contains $Z$ and $K/Z$ is compact.
Let $\rho$ be a finite-dimensional irreducible smooth representation of $K$.
We assume that the representation $\pi:=\cInd_K^{G} \rho$ is irreducible, so supercuspidal.
Then, for every $g \in G^{\rs}$, we have
\[
\Theta_\pi(g)
=\sum_{\begin{subarray}{c} y\in K\backslash G\\ ygy^{-1} \in K \end{subarray}} \tr\rho(ygy^{-1}),
\]
provided that the sum is finite.
\end{thm}

We apply this formula to a computation of the characters of simple supercuspidal representations.

Let $\mathbf{G}$ be $\GL_N$ over $F$.

\begin{lem}\label{lem:Eisen}
Let $g \in I^+ \subseteq \GL_N(F)$ be an affine generic element.
Then $g$ is strongly regular semisimple elliptic. 
\end{lem}

\begin{proof}
The characteristic polynomial of $g-I_N$ is Eisenstein, hence it is irreducible over $F$.
Therefore $g$ is regular semisimple elliptic.
Since the derived group of $\mathbf{G}$ is given by $\SL_{N}$ and in particular simply connected, the regularity of $g$ implies the strong regularity of $g$.
\end{proof}

Now we compute the characters of simple supercuspidal representations at affine generic elements.
We take $a\in k^{\times}$ and $\zeta \in \mu_N$, and consider the simple supercuspidal representation $\pi_{a, \zeta}$ of $G$ defined in Section \ref{sec:ssc}.3.

\begin{prop}\label{prop:charGL}
Let $g \in I^+$ be an affine generic element.
Let $(g_1, \ldots, g_N)$ be the simple affine components of $g$.
Then we have
\[
\Theta_{\pi_{a,\zeta}}(g) = \Kl_{ag_{1} \cdots g_{N}}^{N} (\psi),
\]
where the right-hand side is the Kloosterman sum in Definition $\ref{defn:Kl}$.
\end{prop}

We first prove the following lemma.

\begin{lem}\label{lem:sumGL}
Let $g \in I^+$ be an affine generic element.
If $y \in G$ satisfies $ygy^{-1} \in ZI^+\lan\varphi_{a^{-1}}\ran$, then $y \in ZI\lan\varphi_{a^{-1}}\ran$.
\end{lem}

\begin{proof}
Assume that $y \in G$ satisfies $ygy^{-1} \in ZI^+\lan\varphi_{a^{-1}}\ran$.
As $\det(ygy^{-1})=\det(g)$, $ygy^{-1}$ lies in $I$.
Thus $y \in N_{G}(I)=I\Omega=ZI\lan\varphi_{a^{-1}}\ran$ by Lemma \ref{lem:key}.
\end{proof}

\begin{proof}[Proof of Proposition \ref{prop:charGL}]
By Lemma \ref{lem:sumGL}, if $y \in G$ satisfies $ygy^{-1} \in ZI^+\lan\varphi_{a^{-1}}\ran$, then $y \in ZI\lan\varphi_{a^{-1}}\ran$.
Hence, by the character formula (Theorem \ref{thm:CF}), we have 
\[
\Theta_{\pi_{a,\zeta}}(g) = \sum_{y \in ZI^+\lan\varphi_{a^{-1}}\ran\backslash ZI\lan\varphi_{a^{-1}}\ran} \chi_{a, \zeta}(ygy^{-1}).
\]
Since 
\[
\{t=\diag(t_1, \ldots, t_N) \in T(q) \mid t_N=1 \}
\]
is a system of representatives of the set $ZI^+\lan\varphi_{a^{-1}}\ran\backslash ZI\lan\varphi_{a^{-1}}\ran$, 
we have 
\begin{align*}
\RHS &= \sum_{t_1, \ldots, t_{N-1} \in k^{\times}} \psi\left(\frac{t_1}{t_2}g_1 + \cdots + \frac{t_{N-1}}{1}g_{N-1} + a\cdot\frac{1}{t_1}g_N\right) \\
&= \sum_{\begin{subarray}{c} s_1, \ldots, s_n \in k^{\times}\\ s_1\cdots s_N = a g_1 \cdots g_N \end{subarray}} \psi(s_1 + \cdots + s_N) \\
&= {\rm Kl}_{ag_1 \cdots g_N}^N (\psi).
\end{align*}
\end{proof}

\subsection{The case of the twisted $\GL_{2n}(F)$}
In this subsection, we compute the twisted characters of self-dual simple supercuspidal representations of $\GL_{2n}(F)$ with trivial central characters.

Let us first recall the twisted characters of representations of $p$-adic reductive groups.
We consider a connected reductive group $\mathbf{G}$ over $F$ and an automorphism $\theta$ of $\mathbf{G}$ over $F$.
Then $g\in \mathbf{G}(\ol{F})$ is said to be $\theta$-$\mathit{semisimple}$ if the automorphism $\mathrm{Int}(g)\circ\theta$ of $\mathbf{G}$ is quasi-semisimple (i.e., its restriction to the derived group of $\mathbf{G}$ is semisimple).
Moreover, $\theta$-semisimple element $g\in \mathbf{G}(\ol{F})$ is said to be $\theta$-$\mathit{regular}$ if the identity component $\Cent_{\mathbf{G}}(g, \theta)^{0}$ of the $\theta$-twisted centralizer $\Cent_{\mathbf{G}}(g, \theta)$ of $g$ in $\mathbf{G}$ is a torus.
We write $G^{\trs}$ for the set of $\theta$-regular $\theta$-semisimple elements of $G$.
This is an open subset of $G$.

\begin{thm}[{\cite[Theorem 1]{MR889110}}]\label{thm:tchar}
Let $\mathbf{G}$ be a connected reductive group over $F$.
Let $\theta$ be an automorphism of $\mathbf{G}$ defined over $F$.
Let $\pi$ be a $\theta$-stable $($i.e., $\pi \cong \pi^{\theta}$$)$ irreducible smooth representation of $G$, and fix an isomorphism $A \colon \pi \ra \pi^\theta$.
Then there exists a unique locally constant function $\Theta_{\pi, \theta}$ on $G^{\trs}$ such that 
\[
\tr \pi_\theta (f) = \int_{G^{\trs}} \Theta_{\pi, \theta}(g) f(g)\,dg
\]
for every $f \in \mathcal{H}(G)$ satisfying $\supp(f) \subseteq G^{\trs}$, 
where $\tr \pi_\theta$ is the $\theta$-twisted distribution character of $\pi$ with respect to $A$, that is
\[
\tr \pi_{\theta} (f):= \tr \bigl( \pi(f)\circ A \bigr).
\]
\end{thm}

We call $\Theta_{\pi, \theta}$ the $\theta$-$\mathit{twisted}$ $\mathit{character}$ of $\pi$.
This function is invariant under $\theta$-conjugation, and depends on an isomorphism $A \colon \pi \cong \pi^\theta$ (determined up to a scalar multiple).

Similarly to the standard case, 
we have a formula for the twisted character of supercuspidal representations.

\begin{thm}[Twisted character formula, {\cite[I.6.2 $\mathrm{Th\acute{e}or\grave{e}me}$]{MR3632513}}]\label{thm:TCF}
Let $\mathbf{G}$ be a reductive group over $F$ and $\mathbf{Z}$ its center.
Let $\theta$ be an automorphism of $\mathbf{G}$ over $F$.
Let $K$ be a $\theta$-stable open subgroup of $G$ such that $K$ contains $Z$ and $K/Z$ is compact.
Let $\rho$ be a finite-dimensional $\theta$-stable irreducible smooth representation of $K$.
We fix an isomorphism $A \colon \rho \ra \rho^{\theta}$.
We assume that the representation $\pi:=\cInd_K^{G} \rho$ is irreducible, so supercuspidal.
Then, for $g \in G^{\trs}$, we have
\[
\Theta_{\pi, \theta}(g)
=\sum_{x\in K\backslash G/K}\sum_{\begin{subarray}{c} y\in KxK\\ yg\theta(y)^{-1} \in K \end{subarray}} \tr \bigl(\rho\bigl(yg\theta(y)^{-1}\bigr)\circ A\bigr),
\]
provided that the sum is finite.
Here $\Theta_{\pi, \theta}$ is the $\theta$-twisted character of $\pi$ with respect to the isomorphism $\pi=\cInd_{K}^{G} \rho \ra \cInd_{K}^{G}(\rho^{\theta}) \cong \pi^\theta$ induced by $A$.
\end{thm}

Let $\mathbf{G}$ be $\GL_{2n}$ over $F$.
Let $\theta$ be the automorphism of $\mathbf{G}$ over $F$ defined by
\[
\theta(g)=J\,{}^{t}\!g^{-1}J^{-1} \text{, where } 
J=\begin{pmatrix} &&&1\\&&-1&\\&\adots&&\\(-1)^{2n-1}&&&\end{pmatrix}. 
\]
For $g \in G$, we put 
\[
\mathcal{N}(g) := g\theta(g) \in G.
\]

The automorphism $\theta$ preserves the subgroups $I^+$ and $I^{++}$, and induces the automorphism of $I^+/I^{++} \cong k^{\oplus2n}$ defined by
\[
(x_1, \ldots, x_{2n-1}, x_{2n}) \mapsto 
(x_{2n-1}, \ldots, x_1, x_{2n}).
\]
In particular we have $\psi_a^\theta=\psi_a$ for $a\in k^{\times}$.
On the other hand, by a simple computation, we can check $\theta(\varphi_{a})=-\varphi_{a}^{-1}$.
Therefore 
\[
\pi_{a, \zeta}^{\theta} \cong \cInd_{ZI^+\lan\varphi_{a^{-1}}\ran}^{G} \chi_{a, \zeta}^\theta = \cInd_{ZI^+\lan\varphi_{a^{-1}}\ran}^{G} \chi_{a,\zeta^{-1}} = \pi_{a, \zeta^{-1}},
\] 
and $\pi_{a, \zeta}$ is self-dual if and only if $\zeta=\pm1$.

For $\zeta \in \{\pm1\}$, we take an isomorphism $A\colon \chi_{a,\zeta}\ra\chi_{a,\zeta}^{\theta}=\chi_{a,\zeta}$ to be the identity map.
Then this defines the twisted character $\Theta_{\pi_{a,\zeta},\theta}$ of the simple supercuspidal representation $\pi_{a,\zeta}$ of $G$.

In this subsection, we focus on the case where $a=1$.
First, we compute the twisted character $\Theta_{\pi_{1,\zeta}, \theta}$ at $g \in I^+\cap G^{\trs}$ such that $\mathcal{N}(g)=g\theta(g) \in I^+$ is affine generic.

\begin{lem}\label{lem:sumTGL}
Let $g \in I^+$ be an element such that $\mathcal{N}(g)$ is affine generic.
If $y \in G$ satisfies $yg\theta(y)^{-1} \in ZI^+\lan\varphi_{1}\ran$, then $y \in ZI\lan\varphi_{1}\ran$.
\end{lem}

\begin{proof}
As $yg\theta(y)^{-1} \in ZI^+\lan\varphi_{1}\ran$, $\mathcal{N}(yg\theta(y)^{-1}) = y\mathcal{N}(g)y^{-1}$ is in $\mathcal{N}(ZI^+\lan\varphi_{1}\ran) = ZI^+\lan\varphi_{1}\ran$. 
By the assumption and Lemma \ref{lem:sumGL}, $y$ must lie in $ZI\langle \varphi_{1} \rangle$.
\end{proof}

\begin{lem}\label{lem:sumTGL2}
Let $g \in I^+$ be an element such that $\mathcal{N}(g)$ is affine generic.
Then a system of representatives of the set 
\[
\bigl\{ y \in ZI^+\lan\varphi_{1}\ran\backslash ZI\lan\varphi_{1}\ran \mid yg\theta(y)^{-1} \in ZI^+\lan\varphi_{1}\ran\bigr\}
\]
is given by 
\[
T'(q) := \{ \diag(t_1, \ldots, t_{2n})\in T(q) \mid t_1 t_{2n}=\cdots=t_nt_{n+1}, t_n=1\}.
\]
\end{lem}

\begin{proof} 
Let $y := \diag(t_1, \ldots, t_{2n}) \in T(q)$ satisfying $t_n=1$.
Since 
\[
\mathrm{val}\bigl(\det\bigl(yg\theta(y)^{-1}\bigr)\bigr)=0,
\]
we have
\[
yg\theta(y)^{-1} \in ZI^+\lan\varphi_{1}\ran \Ra yg\theta(y)^{-1} \in Z(q)I^+
\] 
(note that $Z(q)I^+$ is the kernel of $\mathrm{val}\circ\mathrm{det}\colon ZI^+\lan\varphi_{1}\ran\ra\Z$).

Thus we have $yg\theta(y)^{-1} \in ZI^+\lan\varphi_{1}\ran$ if and only if 
the diagonal part of 
\begin{align*}
yg\theta (y)^{-1} &= 
\diag(t_{1}, \ldots, t_{2n})
\begin{pmatrix}
g_{1, 1} & \hdots & g_{1, 2n} \\
\vdots & \ddots & \vdots \\
g_{2n, 1} & \hdots & g_{2n, 2n}
\end{pmatrix}
\diag(t_{2n}, \ldots, t_1)\\
&= 
\begin{pmatrix}
 t_1t_{2n}g_{1, 1}&t_1t_{2n-1}g_{1, 2}&&\\
 &t_2t_{2n-1}g_{2, 2}&\ddots&\ast\\
 &\ast&\ddots&t_{2n-1}t_1g_{2n-1, 2n}\\
 t_{2n}^2g_{2n, 1}&&&t_{2n}t_1g_{2n, 2n}
\end{pmatrix}
\end{align*}
lies in $Z(q)T_1$.
Therefore the condition $yg\theta (y)^{-1} \in ZI^+\lan\varphi_{1}\ran$ is equivalent to the condition $t_1 t_{2n}=\cdots=t_nt_{n+1}$.
\end{proof}

\begin{prop}\label{prop:charTGL}
Let $g \in I^+\cap G^{\trs}$ be an element such that $\mathcal{N}(g)$ is affine generic.
Let $(g_1, \ldots, g_{2n})$ be the simple affine components of $g$.
Then we have 
\[
\Theta_{\pi_{1,\zeta}, \theta}(g) = \Kl_{g_n (g_1 +g_{2n-1})^2 \cdots (g_{n-1}+g_{n+1})^2 g_{2n}}^{n+1} (\psi; 1, 2, \ldots, 2, 1),
\]
where the right-hand side is the Kloosterman sum in Definition $\ref{defn:Kl}$.
\end{prop}

\begin{proof}
Note that the simple affine components of $\mathcal{N}(g)$ are given by 
\[
(g_1+g_{2n-1}, \ldots, g_{2n-1}+g_1, 2g_{2n}),
\]
and the affine genericity of $\mathcal{N}(g)$ means that none of them is zero.

By the twisted character formula (Theorem \ref{thm:TCF}) and Lemma \ref{lem:sumTGL2}, we can compute the twisted character as follows: 
\begin{align*}
&\Theta_{\pi_{1,\zeta}, \theta}(g) 
= \sum_{y \in T'(q)} \chi_{1, \zeta}\bigl(yg\theta(y)^{-1}\bigr) \\
&= \sum_{z \in k^{\times}} \sum_{\begin{subarray}{c} t_1, \ldots, t_n \in k^{\times}\\ t_i t_{2n+1-i} =z, t_n=1 \end{subarray}} 
\psi\left(\frac{t_1 t_{2n-1} g_1}{z} + \cdots + \frac{t_{2n-1} t_1 g_{2n-1}}{z} + \frac{t_{2n}^2 g_{2n}}{z}\right) \\
&= \sum_{z \in k^{\times}} \sum_{t_1, \ldots, t_{n-1} \in k^{\times}} 
\psi \left( \frac{t_1}{t_2}(g_1+g_{2n-1}) + \cdots +\frac{t_{n-1}}{1}(g_{n-1}+g_{n+1}) + \frac{1}{z}g_n + \frac{z}{t_1^2}g_{2n} \right) \\
&= \sum_{t_1, \ldots, t_n \in k^{\times}} 
\psi \left( \frac{t_1}{t_2}(g_1+g_{2n-1}) + \cdots +\frac{t_{n-1}}{1}(g_{n-1}+g_{n+1}) + \frac{t_n}{t_1}g_n + \frac{1}{t_1 t_n}g_{2n} \right) \\
&=\sum_{t_1, \ldots, t_n \in k^{\times}} 
\psi \left( \frac{t_1}{t_2}g_n + \frac{t_2}{t_3}(g_1+g_{2n-1}) + \cdots +\frac{t_n}{1}(g_{n-1}+g_{n+1}) + \frac{1}{t_1 t_2}g_{2n} \right) \\
&= {\rm Kl}_{g_n (g_1 +g_{2n-1})^2 \cdots (g_{n-1}+g_{n+1})^2 g_{2n}}^{n+1} (\psi; 1, 2, \ldots, 2, 1). 
\end{align*}
Here, for the 4th equality we replaced $z$ by $t_1/t_n$, and for the 5th equality we permuted the indices of the $t_{i}$ by $(1 \ 2 \ \cdots \ n)$.
\end{proof}

Next, we compute the twisted character $\Theta_{\pi_{1,\zeta}, \theta}$ at $\varphi_u g$, where $g \in I^+$ and $u \in k^{\times}$ such that $-\mathcal{N}(\varphi_u g)=\varphi_u g\varphi_u^{-1}\theta(g) \in I^+$ is affine generic.

\begin{lem}\label{lem:sumTGL'}
Let $g \in I^+$ be an element such that $-\mathcal{N}(\varphi_u g) \in I^+$ is affine generic.
If $y \in G$ satisfies $y\varphi_u g\theta(y)^{-1} \in ZI^+\lan\varphi_{1}\ran$, then $y \in ZI\lan\varphi_{1}\ran$.
\end{lem}

\begin{proof}
Since $y\varphi_u g\theta(y)^{-1} \in ZI^+\lan\varphi_{1}\ran$, $\mathcal{N}(y\varphi_u g\theta(y)^{-1}) = y\mathcal{N}(\varphi_u g)y^{-1}$ belongs to $\mathcal{N}(ZI^+\lan\varphi_{1}\ran)=ZI^+\lan\varphi_{1}\ran$.
By the assumption and Lemma \ref{lem:sumGL}, $y$ must lie in $ZI\langle \varphi_{1} \rangle$. 
\end{proof}

\begin{lem}\label{lem:sumTGL'2}
Let $g \in I^+$ be an element such that $-\mathcal{N}(\varphi_u g) \in I^+$ is affine generic.
Then a system of representatives of the set 
\[
\bigl\{ y \in ZI^+\lan\varphi_{1}\ran\backslash ZI\lan\varphi_{1}\ran \mid y\varphi_u g\theta (y)^{-1} \in ZI^+\lan\varphi_{1}\ran\bigr\}
\]
is given by 
\[
T''(q) := \{ \diag(t_1, \ldots, t_{2n})\in T(q) \mid t_1 t_{2n-1}= \cdots = t_{n-1} t_{n+1} = t_{n}^{2} = u, t_{2n}=1 \}.
\]
\end{lem}

\begin{proof}
Let $y = \diag(t_1, \ldots, t_{2n}) \in T(q)$ satisfying $t_{2n}=1$.
Since 
\[
\mathrm{val}\bigl(\det\bigl(y\varphi_u g\theta (y)^{-1}\bigr)\bigr)=\mathrm{val}(\det(\varphi_{1})),
\]
we have
\[
y\varphi_u g\theta (y)^{-1} \in ZI^+\lan\varphi_{1}\ran \Ra \varphi_{1}^{-1}y\varphi_u g\theta (y)^{-1} \in Z(q)I^+.
\] 

Thus we have $y\varphi_u g\theta (y)^{-1} \in ZI^+\lan\varphi_{1}\ran$ if and only if the diagonal part of
\[
\varphi_{1}^{-1}y\varphi_u \cdot g \cdot \theta (y)^{-1} 
\]
\begin{align*}
&=\diag(t_{2n}u, t_1,\ldots, t_{2n-1})
\begin{pmatrix}
g_{1, 1}&\hdots&g_{1, 2n}\\
\vdots&\ddots&\vdots\\
g_{2n, 1}&\hdots&g_{2n, 2n}
\end{pmatrix}
\diag(t_{2n}, t_{2n-1}, \ldots, t_1)\\
&= 
\begin{pmatrix}
 t_{2n}^2ug_{1, 1}&t_{2n}t_{2n-1}ug_{1, 2}&&\\
 &t_1t_{2n-1}g_{2, 2}&\ddots&\ast\\
 &\ast&\ddots&t_{2n-2}t_1g_{2n-1, 2n}\\
 t_{2n-1}t_{2n}g_{2n, 1}&&&t_{2n-1}t_1g_{2n, 2n}
\end{pmatrix}
\end{align*}
lies in $Z(q)T_1$.
Therefore the condition $y\varphi_u g\theta (y)^{-1} \in ZI^+\lan\varphi_{1}\ran$ is equivalent to the condition $t_1 t_{2n-1}= \cdots = t_{n-1} t_{n+1} = t_{n}^{2} = u$.
\end{proof}

\begin{prop}\label{prop:charTGL'}
Let $g \in I^+$ be an element such that $\varphi_ug \in G^{\trs}(F)$ and $-\mathcal{N}(\varphi_u g)$ is affine generic.
Let $(g_1, \ldots, g_{2n})$ be the simple affine components of $g$.
\begin{enumerate}
\item If $u\notin k^{\times2}$, then we have 
\[
\Theta_{\pi_{1,\zeta}, \theta} (\varphi_u g)=0.
\]
\item If $u=v^2$ for some $v \in k^{\times}$, then we have 
\[
\Theta_{\pi_{1,\zeta}, \theta} (\varphi_u g) = \zeta \left( \Kl_{(v^2g_1+g_{2n})\cdots(g_n+g_{n+1})/v}^{n}(\psi) + \Kl_{-(v^2g_1+g_{2n})\cdots(g_n+g_{n+1})/v}^{n}(\psi) \right).
\]
\end{enumerate}
\end{prop}

\begin{proof}
Note that the simple affine components of $-\mathcal{N}(\varphi_u g)$ are given by 
\[
(g_2+g_{2n-1}, g_3+g_{2n-2}, \ldots, g_{2n-1}+g_2, u^{-1}g_{2n}+g_1, ug_1+g_{2n}),
\]
and the affine genericity of $-\mathcal{N}(\varphi_{u}g)$ means that none of them is zero.

We use the twisted character formula (Theorem \ref{thm:TCF}) and Lemma \ref{lem:sumTGL'2}.
If $u\notin k^{\times2}$, then the set $T''(q)$ is empty.
Hence the sum in the twisted character formula is zero.

If $u=v^2$ for some $v\in k^{\times}$, then we can compute the twisted character as follows: 
\begin{align*}
& \Theta_{\pi_{1,\zeta}, \theta} (\varphi_{v^2} g) 
= \sum_{y \in T''(q)} \chi_{1, \zeta}(\varphi_{1})\chi_{1, \zeta}\bigl(\varphi_{1}^{-1}y\varphi_{v^2} g\theta(y)^{-1}\bigr) \\
&= \zeta \sum_{\begin{subarray}{c} t_1, \ldots, t_{2n} \in k^{\times}\\ t_i t_{2n-i} =v^2\\ t_{2n} = 1 \end{subarray}} 
\psi \left( \frac{t_{2n}t_{2n-1}v^2g_1}{v^2} + \frac{t_1t_{2n-2}g_2}{v^2} + \cdots + \frac{t_{2n-2}t_1g_{2n-1}}{v^2} + \frac{t_{2n-1}t_{2n}g_{2n}}{v^2} \right) \\
&= \zeta \sum_{\begin{subarray}{c} t_1, \ldots, t_{n} \in k^{\times}\\ t_n=\pm v \end{subarray}} 
\psi \left( \frac{1}{t_1}(v^2g_1+g_{2n}) + \frac{t_1}{t_2}(g_2+g_{2n-1}) + \cdots + \frac{t_{n-1}}{t_n}(g_n+g_{n+1}) \right) \\
&= \zeta \left( {\rm Kl}_{(v^2g_1+g_{2n})\cdots(g_n+g_{n+1})/v}^{n}(\psi) + {\rm Kl}_{-(v^2g_1+g_{2n})\cdots(g_n+g_{n+1})/v}^{n}(\psi) \right).
\end{align*}
\end{proof}

\subsection{The case of $\SO_{2n+1}(F)$}
Let $\mathbf{H}$ be $\SO_{2n+1}$.
We take $b \in k^{\times}$ and $\xi \in \{\pm1\}$ and consider the simple supercuspidal representation $\pi'_{b,\xi}$ of $H$ defined in Section \ref{sec:ssc}.4.

First, we compute the character of $\pi'_{b,\xi}$ at an affine generic element $h \in I_H^+\cap H^{\rs}$.

\begin{lem}\label{lem:sumSO}
Let $h \in I_H^+$ be an affine generic element.
If $y \in H$ satisfies $yhy^{-1} \in I_H^+\lan\varphi'_{b^{-1}}\ran$, then $y \in I_H\lan\varphi'_{b^{-1}}\ran$.
\end{lem}

\begin{proof}
If $yhy^{-1} \in I_H^+\lan\varphi'_{b^{-1}}\ran$, then $(yhy^{-1})^2 = yh^2y^{-1} \in I_H^+$.
Since we assumed that the characteristic of $k$ is not equal to $2$, 
$h^2$ is also affine generic.
Therefore $y \in I_H\lan\varphi'_{b^{-1}}\ran$ by Lemma \ref{lem:key}.
\end{proof}

\begin{prop}\label{prop:charSO}
Let $h \in I_H^+\cap H^{\rs}$ be an affine generic element with its simple affine components $(h_1, \ldots, h_n, h_{2n})$.
Then we have 
\[
\Theta_{\pi'_{b,\xi}}(h) = \Kl_{bh_1 h_2^2\cdots h_n^2 h_{2n}}^{n+1} (\psi; 1, 2, \ldots, 2, 1).
\]
\end{prop}

\begin{proof}
By the character formula (Theorem \ref{thm:CF}) and Lemma \ref{lem:sumSO}, we can compute the character as follows: 
\begin{align*}
\Theta_{\pi'_{b,\xi}}(h) &= \sum_{y \in I_H^+ \langle \varphi'_{b^{-1}} \rangle \backslash I_H \langle \varphi'_{b^{-1}} \rangle} \chi'_{b, \xi}(yhy^{-1}) 
= \sum_{t \in T_H(q)} \chi'_{b, \xi}(tht^{-1}) \\
&= \sum_{t_1, \ldots, t_n \in k^{\times}} \psi \left( \frac{t_1}{t_2}h_1 + \cdots + \frac{t_{n-1}}{t_n}h_{n-1} + \frac{t_n}{1}h_n + \frac{b}{t_1 t_2}h_{2n} \right) \\
&= \sum_{\begin{subarray}{c} s_1, \ldots, s_n, s_{2n} \in k^{\times}\\ s_1 s_2^2 \cdots s_n^2 s_{2n} = b h_1 h_2^2\cdots h_n^2 h_{2n}\end{subarray}} \psi(s_1 + \cdots + s_n+s_{2n}) \\
&= {\rm Kl}_{b h_1 h_2^2\cdots h_n^2 h_{2n}}^{n+1} (\psi; 1, 2, \ldots, 2, 1). 
\end{align*}
\end{proof}

\begin{rem}
Later, we will prove that every affine generic element of $H$ is strongly regular semisimple (Proposition \ref{prop:affgennorm}).
\end{rem}

Next, we compute the character of $\pi'_{b,\xi}$ at $\varphi'_{u}h$, where $h \in I_H^+$ and $u \in k^{\times}$ such that $(\varphi'_{u}h)^2 \in I_H^+$ is affine generic.

\begin{lem}\label{lem:sumSO'}
Let $h \in I_H^+$ be an element such that $(\varphi'_{u}h)^2$ is affine generic.
If $y \in H$ satisfies $y\varphi'_{u}hy^{-1} \in I_H^+\lan\varphi'_{b^{-1}}\ran$, then $y \in I_H\lan\varphi'_{b^{-1}}\ran$. 
\end{lem}

\begin{proof}
If $y\varphi'_{u}hy^{-1} \in I_H^+\lan\varphi'_{b^{-1}} \ran$, then $(y\varphi'_{u}hy^{-1})^2 = y(\varphi'_{u}h)^2y^{-1} \in  I_H^+$.
By the assumption and Lemma \ref{lem:key}, $y$ belongs to $I_H\lan\varphi' _{b^{-1}}\ran$. 
\end{proof}

\begin{lem}\label{lem:sumSO'2}
Let $h \in I_H^+$ be an element such that $(\varphi'_{u}h)^2$ is affine generic.
Then the set 
\[
\bigl\{y \in I_H^+\lan\varphi'_{b^{-1}}\ran\backslash I_H\lan\varphi'_{b^{-1}}\ran\ \mid y\varphi'_{u}hy^{-1} \in  I_H^+\lan\varphi'_{b^{-1}} \ran\bigr\}
\]
is represented by 
\[
T'_H(q) := \bigl\{ \diag(t_1, \ldots, t_n, 1, t_n^{-1}, \ldots, t_1^{-1})\in T_H(q) \mid t_1^2=bu \bigr\}.
\]
\end{lem}

\begin{proof}
Let $y = \diag(t_1, \ldots, t_n, 1, t_n^{-1}, \ldots, t_1^{-1})\in T_H(q)$ be an element such that $y\varphi'_{u}hy^{-1} \in  I_H^+\langle \varphi'_{b^{-1}} \rangle$.
We have a decomposition $I_H^+\langle \varphi'_{b^{-1}} \rangle = I_H^+ \sqcup \varphi'_{b^{-1}}I_H^+$, 
and we can check $y\varphi'_{u}hy^{-1} \notin I_H^+$, hence $y\varphi'_{u}hy^{-1}$ belongs to $\varphi'_{b^{-1}}I_H^+$.
Thus we have $y\varphi'_{u}hy^{-1} \in  I_H^+\lan\varphi'_{b^{-1}} \ran$ if and only if 
the diagonal part of 
\[
(\varphi'_{b^{-1}})^{-1}y\varphi'_{u}\cdot h \cdot y^{-1}
\]
\begin{align*}
&=\diag(t_1^{-1}bu, t_2, t_3, \ldots)
\begin{pmatrix}
 h_{1, 1}&\hdots&h_{1, 2n+1}\\
 \vdots&\ddots &\vdots\\
 h_{2n+1, 1}&\hdots &h_{2n+1, 2n+1}
\end{pmatrix}
\diag(t_1^{-1}, t_2^{-1}, t_3^{-1}, \ldots)\\
&= 
\begin{pmatrix}
 t_1^{-2}buh_{1, 1}&t_1^{-1}t_2^{-1}buh_{1, 2}&&\\
 &h_{2, 2}&\ddots&\ast\\
  t_1^{-1}t_2^{-1}h_{2n, 1}&\ast&\ddots&t_1t_2^{-1}h_{2n, 2n+1}\\
 &t_1t_2^{-1}(bu)^{-1}h_{2n+1, 2}&&t_1^2(bu)^{-1}h_{2n+1, 2n+1}
\end{pmatrix}
\end{align*}
lies in $(T_H)_{1}$.
Therefore $y\varphi'_{u}hy^{-1} \in  I_H^+\lan\varphi'_{b^{-1}}\ran$ is equivalent to $t_1^2=bu$.
\end{proof}

\begin{prop}\label{prop:charSO'}
Let $h \in I_H^+$ be an element such that $\varphi'_{u}h \in H^{\rs}$ and $(\varphi'_{u}h)^2$ is affine generic.
Let $(h_1, \ldots, h_n, h_{2n})$ be the simple affine components of $h$.
\begin{enumerate}
 \item If $bu \notin k^{\times2}$, then we have 
 \[
 \Theta_{\pi'_{b,\xi}}(\varphi'_{u}h)=0. 
 \]
 \item If $bu=v^2$ for some $v \in k^{\times}$, then we have
 \[
 \Theta_{\pi'_{b,\xi}}(\varphi'_{u}h) = \xi \left( {\rm Kl}_{(h_1v^2+bh_{2n})h_2\cdots h_n/v}^{n}(\psi) + {\rm Kl}_{-(h_1v^2+bh_{2n})h_2\cdots h_n/v}^{n}(\psi) \right).
 \]
\end{enumerate}
\end{prop}

\begin{proof}
Note that the simple affine components of $(\varphi'_{u}h)^2$ are given by 
\[
(h_{2n}u^{-1}+h_1, 2h_2, \ldots, 2h_n, h_1u+h_{2n}),
\]
and the affine genericity of $(\varphi_{u}'h)^{2}$ means that none of them is zero.

We use the character formula (Theorem \ref{thm:CF}) and Lemma \ref{lem:sumSO'2}.
If $bu\notin k^{\times2}$, then the set $T'_H(q)$ is empty.
Hence the sum in the character formula is zero.

If $bu=v^2$ for some $v\in k^{\times}$, then we can compute the character as follows: 
\begin{align*}
\Theta_{\pi'_{b,\xi}}(\varphi'_{b^{-1}v^2}h) 
&= \sum_{y \in T'_H(q)} \chi'_{b, \xi}(\varphi'_{b^{-1}})\chi'_{b, \xi}\bigl((\varphi'_{b^{-1}})^{-1}y\varphi'_{b^{-1}v^2}hy^{-1}\bigr) \\
&= \xi \sum_{\begin{subarray}{c} t_1, \ldots, t_{n} \in k^{\times}\\ t_1 = \pm v \end{subarray}} 
\psi \left( \frac{1}{t_1t_2}h_1v^2 + \frac{t_2}{t_3}h_2 + \cdots + \frac{t_n}{1}h_n + \frac{b}{t_1t_2}h_{2n} \right) \\
&= \xi \left( {\rm Kl}_{(h_1v^2+bh_{2n})h_2\cdots h_n/v}^{n}(\psi) + {\rm Kl}_{-(h_1v^2+bh_{2n})h_2\cdots h_n/v}^{n}(\psi) \right).
\end{align*} 
\end{proof}

\section{Norm correspondences}\label{sec:norm}
\subsection{Norm correspondences}

Let us first recall the norm correspondence for twisted endoscopy.
Our basic reference in this section is \cite{MR1687096}.

Let $\mathbf{G}$ be a connected quasi-split reductive group over $F$, and $\theta$ an automorphism of $\mathbf{G}$ defined over $F$.
Let $(\mathbf{H}, {}^{L}\mathbf{H}, s, \xi)$ be a quadruple of endoscopic data for the triple $(\mathbf{G}, \theta, 1)$.
Then we have the map 
\[
\mathcal{A}_{\mathbf{H}/\mathbf{G}} \colon Cl_\mathrm{ss}(\mathbf{H}) \ra Cl_{\theta\mathchar`-\mathrm{ss}}(\mathbf{G}, \theta) 
\]
from the set of semisimple conjugacy classes in $\mathbf{H}(\ol{F})$ to the set of $\theta$-semisimple $\theta$-conjugacy classes in $\mathbf{G}(\ol{F})$ (see Section 3.3 in \cite{MR1687096}).

A $\theta$-semisimple element $g\in \mathbf{G}(\ol{F})$ is said to be $\mathit{strongly}$ $\theta$-$\mathit{regular}$ if the $\theta$-centralizer $\Cent_{\mathbf{G}}(g, \theta)$ of $g$ in $\mathbf{G}$ is abelian.
Let $G^{\strs}$ be the set of strongly $\theta$-regular $\theta$-semisimple elements in $G$, and $H^{\srs}$ the set of strongly regular semisimple elements in $H$.
We say that $y\in H^{\srs}$ is a $\mathit{norm}$ of $x\in G^{\strs}$ if $x$ corresponds to $y$ via the map $\mathcal{A}_{\mathbf{H}/\mathbf{G}}$.

In this paper, we consider the following situation:
\begin{itemize}
 \item $\mathbf{G}=\GL_{2n}$ over $F$, 
 \item $\theta(g)=J\,{}^{t}\!g^{-1}J^{-1}$, where \[J=\begin{pmatrix} &&&1\\&&-1&\\&\adots&&\\(-1)^{2n-1}&&&\end{pmatrix},\] 
 \item $\mathbf{H}=\SO_{2n+1}$ over $F$, 
 \item $s=1$, and 
 \item $\xi \colon {}^L\!\SO_{2n+1}=\Sp_{2n}(\C)\times W_F \hookrightarrow \GL_{2n}(\C)\times W_F={}^L\!\GL_{2n}$.
\end{itemize}
For $g \in \mathbf{G}(\ol{F})$, we put 
\[
\mathcal{N}(g):=g\theta(g) \in \mathbf{G}(\ol{F}).
\]

We can write the map $\mathcal{A}_{\mathbf{H}/\mathbf{G}}$ explicitly in terms of the diagonal tori $\mathbf{T}$ and $\mathbf{T}_{\mathbf{H}}$ as follows (note that every $\theta$-semisimple element is $\theta$-conjugate to an element of $\mathbf{T}(\ol{F})$, see \cite[Lemma 3.2.A]{MR1687096})
:
\[\xymatrix{
 Cl_\mathrm{ss}(\mathbf{H})& 
 \mathbf{T}_{\mathbf{H}}(\ol{F})/\Omega_{\mathbf{T}_{\mathbf{H}}} \ar[l]_-{\cong} \ar[r]^-{\cong}& 
 \mathbf{T}_{\theta}(\ol{F})/\Omega_{\mathbf{T}}^\theta \ar^-{\cong}[r]& 
 Cl_{\theta\mathrm{\mathchar`-ss}}(\mathbf{G}, \theta) &
 \mathbf{T}(\ol{F}) \ar@{->>}[l]
}\]
\[
\diag\left(\frac{t_1}{t_{2n}}, \ldots, \frac{t_n}{t_{n+1}}, 1, \frac{t_{n+1}}{t_n}, \ldots, \frac{t_{2n}}{t_1}\right) \longleftrightarrow \diag(t_1, \ldots, t_{2n}),
\]
where $\Omega_{\mathbf{T}_{\mathbf{H}}}$ is the Weyl group of $\mathbf{T}_{\mathbf{H}}$ in $\mathbf{H}$, 
$\Omega_{\mathbf{T}}^{\theta}$ is the $\theta$-fixed part of the Weyl group of $\mathbf{T}$ in $\mathbf{G}$, and 
$\mathbf{T}_{\theta}$ is the $\theta$-coinvariant of $\mathbf{T}$.
Here note that the map $\mathcal{A}_{\mathbf{H}/\mathbf{G}}$ is an isomorphism 
since $\mathbf{T}_{\mathbf{H}}(\ol{F}) \cong \mathbf{T}_{\theta}(\ol{F})$ and $\Omega_{\mathbf{T}_{\mathbf{H}}} \cong \Omega_{\mathbf{T}}^{\theta}$ in our situation.
As a consequence of this explicit description of the map $\mathcal{A}_{\mathbf{H}/\mathbf{G}}$, we get the following two lemmas.

\begin{lem}\label{lem:eig}
For a $\theta$-semisimple element $g$ of $\mathbf{G}(\ol{F})$ and a semisimple element $h$ of $\mathbf{H}(\ol{F})$, 
let $\mathrm{Eig}(\mathcal{N}(g))$ and $\mathrm{Eig}(h)$ be the multi-sets consisting of eigenvalues of $\mathcal{N}(g)=g\theta(g)$ and $h$, respectively.
Then $h$ corresponds to $g$ via $\mathcal{A}_{\mathbf{H}/\mathbf{G}}$ if and only if we have
\[
\mathrm{Eig}\bigl(\mathcal{N}(g)\bigr)\sqcup\{1\}=\mathrm{Eig}(h).
\]
\end{lem}

\begin{proof}
We take $t=\diag(t_1, \ldots, t_{2n})\in \mathbf{T}(\ol{F})$ which is  $\theta$-conjugate to $g$ in $\mathbf{G}(\ol{F})$, and $s=\diag(s_1, \ldots, s_{n}, 1, s_{n}^{-1}, \ldots, s_{1}^{-1})\in \mathbf{T}_{\mathbf{H}}(\ol{F})$ which is conjugate to $h$ in $\mathbf{H}(\ol{F})$.
Then, by the above description of $\mathcal{A}_{\mathbf{H}/\mathbf{G}}$, $h$ corresponds to $g$ if and only if $s$ is conjugate to 
\[
\diag\left(\frac{t_1}{t_{2n}}, \ldots, \frac{t_n}{t_{n+1}}, 1, \frac{t_{n+1}}{t_n}, \ldots, \frac{t_{2n}}{t_1}\right)
\]
in $\mathbf{H}(\ol{F})=\SO_{2n+1}(\ol{F})$.
However, this is equivalent to saying that they are conjugate in $\GL_{2n+1}(\ol{F})$ (see, for example, \cite[IV.2.15]{MR0268192}), hence equivalent to saying that we have the following equality of multi-sets:
\[
\biggl\{\frac{t_1}{t_{2n}}, \ldots, \frac{t_n}{t_{n+1}}, 1, \frac{t_{n+1}}{t_n}, \ldots, \frac{t_{2n}}{t_1}\biggr\}=\{s_1, \ldots, s_{n}, 1, s_{n}^{-1}, \ldots, s_{1}^{-1}\}=\mathrm{Eig}(h).
\]

If we take $x\in \mathbf{G}(\ol{F})$ such that $xg\theta(x)^{-1}=t$, then we have
\[
\mathcal{N}\bigl(xg\theta(x)^{-1}\bigr)=xg\theta(x)^{-1}\cdot\theta\bigl(xg\theta(x)^{-1}\bigr)=x\mathcal{N}(g)x^{-1}.
\]
On the other hand, we have 
\[
\mathcal{N}(t)=\diag\biggl(\frac{t_1}{t_{2n}}, \ldots, \frac{t_n}{t_{n+1}}, \frac{t_{n+1}}{t_n}, \ldots, \frac{t_{2n}}{t_1}\biggr).
\]
Thus the multi-set $\mathrm{Eig}(\mathcal{N}(g))$ is given by
\[
\biggl\{\frac{t_1}{t_{2n}}, \ldots, \frac{t_n}{t_{n+1}}, \frac{t_{n+1}}{t_n}, \ldots, \frac{t_{2n}}{t_1}\biggr\},
\]
and we get the assertion.
\end{proof}

\begin{lem}\label{lem:norm1}
Let $g$ be a $\theta$-semisimple element in $\mathbf{G}(\ol{F})$ and $h$ a semisimple element in $\mathbf{H}(\ol{F})$.
If $h$ is conjugate to a matrix 
\[
\begin{pmatrix}
 1&\ast\\
 0&\mathcal{N}(g)
\end{pmatrix}
\]
in $\GL_{2n+1}(\ol{F})$, then $h$ corresponds to $g$ via $\mathcal{A}_{\mathbf{H}/\mathbf{G}}$.
\end{lem}

\begin{proof}
This follows from Lemma \ref{lem:eig} immediately.
\end{proof}

Let us prove some lemmas needed later.

\begin{lem}\label{lem:norm0}
Let $g$ be a $\theta$-semisimple element in $\mathbf{G}(\ol{F})$ and $h$ a semisimple element in $\mathbf{H}(\ol{F})$ which corresponds to $h$ via $\mathcal{A}_{\mathbf{H}/\mathbf{G}}$.
If $\mathcal{N}(g)=g\theta(g) \in \mathbf{G}(\ol{F})$ is strongly regular semisimple, then $g$ is strongly $\theta$-regular and $h$ is strongly regular.
\end{lem}

\begin{proof}
As $\mathcal{N}(g)$ is strongly regular semisimple, the centralizer $\Cent_{\mathbf{G}}(\mathcal{N}(g))$ of $\mathcal{N}(g)$ in $\mathbf{G}$ is a maximal torus of $\mathbf{G}$, in particular abelian.
Since $\Cent_{\mathbf{G}}(g, \theta)$ is contained in $\Cent_{\mathbf{G}}(\mathcal{N}(g))$, $\Cent_{\mathbf{G}}(g, \theta)$ is abelian.
Thus $g$ is strongly $\theta$-regular.
By the assumption that $h$ corresponds to $g$, $h$ is also strongly regular (\cite[Lemma 3.3.C.]{MR1687096}). 
\end{proof}

\begin{lem}\label{lem:uniquenorm}
Let $g \in G^{\strs}$.
Then $g$ has at most one norm in $H^{\srs}$ up to stable conjugacy.
\end{lem}

\begin{proof}
Let $h, h' \in H^{\srs}$ be norms of $g$.
In our situation, the map $\mathcal{A}_{\mathbf{H}/\mathbf{G}}$ is bijective, hence $h$ and $h'$ are conjugate in $\mathbf{H}(\ol{F})$.
Since $h$ and $h'$ are strongly regular, they are stably conjugate.
\end{proof}

%

\begin{lem}\label{lem:norm2}
Let $h$ be a strongly regular semisimple elliptic element in $H$.
Then there exists a strongly $\theta$-regular $\theta$-semisimple $\theta$-elliptic element $g \in G$ such that $h$ is a norm of $g$.
\end{lem}

\begin{proof}
This follows from the adjoint relation of the transfer factor for strongly regular semisimple elliptic elements (see the proof of \cite[Proposition 2.1.1]{MR3135650}).
\end{proof}

\subsection{Norm correspondences for affine generic elements}
In this subsection, we study the norm correspondence for affine generic elements.
We use the same notations as in Subsections 2.3 and 2.4.

\begin{lem}\label{lem:partialdiag}
Let $h \in I_H^+ \subseteq H$ be an affine generic element with its simple affine components $(h_1, \ldots, h_n, h_{2n})$.
Then $h$ is conjugate to a matrix 
\[
\begin{pmatrix}
 1 & \ast \\
 0 & h'
\end{pmatrix}
\]
in $\GL_{2n+1}(F)$, where $h'$ is an affine generic element of $G$ with its simple affine components $(h_2, \ldots, h_n, h_n, \ldots, h_1, 2h_{2n})$.
\end{lem}

\begin{proof}
Let 
\[
h := 
\begin{pmatrix}
 x_{1, 1}&\hdots&x_{1, 2n+1}\\
 \vdots&\ddots&\vdots\\
 x_{2n+1, 1}&\hdots&x_{2n+1, 2n+1}
\end{pmatrix} 
\in I_H^+
\]
be an affine generic element with its simple affine components $(h_1, \ldots, h_n, h_{2n})$, then we have
\begin{align*}
x_{1, 2} &\equiv x_{2n, 2n+1} \equiv h_1 \not\equiv 0 \pmod\mfp, \\
&\vdotswithin{=} \\
x_{n, n+1} &\equiv x_{n+1, n+2} \equiv h_n \not\equiv 0 \pmod\mfp \text{, and} \\
x_{2n, 1} \varpi^{-1} &\equiv x_{2n+1, 2} \varpi^{-1} \equiv h_{2n} \not\equiv 0 \pmod\mfp.
\end{align*}

Since $h$ is an element of $\SO_{2n+1}(F)$, $h$ has an eigenvector with eigenvalue $1$.
We take such an eigenvector satisfying 
\[
v=
\begin{pmatrix}
 v_1\\
 \vdots\\
 v_{2n+1}
\end{pmatrix} 
\in \mcO^{\oplus(2n+1)} \setminus \mfp^{\oplus(2n+1)}.
\]

Then we have 
\[
\begin{pmatrix}
 x_{1, 1}-1&\hdots&x_{1, 2n+1}\\
 \vdots&\ddots&\vdots\\
 x_{2n+1, 1}&\hdots&x_{2n+1, 2n+1}-1
\end{pmatrix}
\begin{pmatrix}
 v_1\\
 \vdots\\
 v_{2n+1}
\end{pmatrix} 
= 0 . \tag{$\ast$}
\]
In particular, this gives  
\[
\begin{pmatrix}
 0& x_{1, 2}&\hdots &x_{1, 2n+1}\\
 &\ddots&\ddots&\vdots\\
 &&\ddots&x_{2n, 2n+1}\\
 &\mbox{\Large 0}&&0
\end{pmatrix}
\begin{pmatrix}
v_1\\
\vdots\\
v_{2n+1}
\end{pmatrix} 
\equiv 0 \pmod \mfp.
\]
Hence we have 
\[
v_2 \equiv \cdots \equiv v_{2n+1} \equiv 0 \pmod \mfp.
\]
As $v \notin \mfp^{\oplus(2n+1)}$, we may assume that $v_1=1$.

Consider the following matrix: 
\[
\begin{pmatrix}
 1&&&\\
 v_2&\ddots&\mbox{\Large 0}&\\
 \vdots&&\ddots&\\
 v_{2n+1}&\mbox{\Large 0}&&1
\end{pmatrix}^{-1}
h
\begin{pmatrix}
 1&&&\\
 v_2&\ddots&\mbox{\Large 0}&\\
 \vdots&&\ddots&\\
 v_{2n+1}&\mbox{\Large 0}&&1
\end{pmatrix}
\]
\[
= 
\begin{pmatrix}
 1&x_{1, 2}&\hdots&x_{1, 2n+1}\\
 0&x_{2, 2}-v_2x_{1, 2}&\hdots&x_{2, 2n+1}-v_2x_{1, 2n+1}\\
 \vdots&\vdots&\ddots&\vdots\\
 0&x_{2n+1, 2}-v_{2n+1}x_{1, 2}&\hdots&x_{2n+1, 2n+1}-v_{2n+1}x_{1, 2n+1}\\
\end{pmatrix}.
\]
Then it suffices to show that the matrix
\[
h':=
\begin{pmatrix}
 x_{2, 2}-v_2x_{1, 2}&\hdots&x_{2, 2n+1}-v_2x_{1, 2n+1}\\
 \vdots&\ddots&\vdots\\
 x_{2n+1, 2}-v_{2n+1}x_{1, 2}&\hdots&x_{2n+1, 2n+1}-v_{2n+1}x_{1, 2n+1}\\
\end{pmatrix}
\]
is a desired affine generic element in $G$.
As $h'$ is an element of $I^+$, our task is to compute the simple affine components of $h'$. 

First, as 
\[
v_2 \equiv \cdots \equiv v_{2n+1} \equiv 0 \pmod \mfp,
\]
the first $(2n-1)$-simple affine components of $h'$ are $(h_2, \ldots, h_n, h_n, \ldots, h_1)$.

Second, by the $(2n)$-th row of the equation $(\ast)$, we have 
\[
x_{2n, 1} + x_{2n, 2}v_2 + \cdots + (x_{2n, 2n}-1)v_{2n} + x_{2n, 2n+1}v_{2n+1} = 0.
\]
Thus we have 
\[
x_{2n, 1} + x_{2n, 2n+1}v_{2n+1} \equiv 0 \pmod {\mfp^2},
\]
so the last simple affine component of $h'$ is 
\begin{align*}
(x_{2n+1, 2}-v_{2n+1}x_{1, 2})\varpi^{-1} &\equiv h_{2n}-v_{2n+1}\varpi^{-1}h_1\\
&\equiv 2h_{2n} \pmod {\mfp}.
\end{align*}
\end{proof}

\begin{prop}\label{prop:affgennorm}
Let $h \in I_H^+ \subseteq H$ be an affine generic element with its simple affine components $(h_1, \ldots, h_n, h_{2n})$.
Then $h$ is strongly regular semisimple elliptic, and there exists $g \in G$ satisfying the following conditions:
\begin{itemize}
 \item $g$ is strongly $\theta$-regular $\theta$-semisimple $\theta$-elliptic, 
 \item $h$ is a norm of $g$, and 
 \item $\mathcal{N}(g)$ is an affine generic element of $G$ with simple affine components $(h_2, \ldots, h_n, h_n, \ldots, h_1, 2h_{2n})$.  
\end{itemize}
\end{prop}

\begin{proof}
From Lemma \ref{lem:partialdiag} and the proof of Lemma \ref{lem:Eisen}, it follows that $h$ has $(2n+1)$-distinct eigenvalues.
Hence $h$ is semisimple.
Since the centralizer of $h$ in $H$ is compact by Lemma \ref{lem:sumSO}, $h$ is elliptic.

We next show that $h$ is strongly regular.
Let $h'$ be the affine generic element of $G$ defined in Lemma \ref{lem:partialdiag}.
We take a $\theta$-semisimple element $g' \in \mathbf{G}(\ol{F})$ corresponding to $h$ via $\mathcal{A}_{\mathbf{H}/\mathbf{G}}$.
Then, by Lemma \ref{lem:eig}, we have
\[
\mathrm{Eig}\bigl(\mathcal{N}(g')\bigr)=\mathrm{Eig}(h').
\]
In particular, $\mathcal{N}(g')$ and $h'$ are conjugate in $\mathbf{G}(\ol{F})$.
Since $h'$ is an affine generic element of $G$, $h'$ is a strongly regular semisimple element of $\mathbf{G}(\ol{F})$ by Lemma \ref{lem:Eisen}.
Thus so is $\mathcal{N}(g')$.
By applying Lemma \ref{lem:norm0} to $g'\in \mathbf{G}(\ol{F})$ and $h \in H$, we can conclude that $h$ is a strongly regular element of $H$.


Finally, we take a strongly $\theta$-regular $\theta$-semisimple $\theta$-elliptic element $g'' \in G$ such that $h$ is a norm of $g''$.
Such an element exists by Lemma \ref{lem:norm2}.
Then, by the same argument as above, $\mathcal{N}(g'')$ and $h'$ are conjugate in $\mathbf{G}(\ol{F})$. 
However, since $\mathcal{N}(g'')$ and $h'$ belong to $G=\mathbf{G}(F)$, they are conjugate in $G$.
Let $x'$ be an element of $G$ satisfying $x'\mathcal{N}(g'')x'^{-1}=h'$.
Then $x'g''\theta(x')^{-1}$ belongs to $G$ and is strongly $\theta$-regular $\theta$-semisimple $\theta$-elliptic.
Moreover $h$ is a norm of $x'g''\theta(x')^{-1}$.
Finally, since we have 
\[
\mathcal{N}(x'g''\theta(x')^{-1})=x'\mathcal{N}(g''){x'}^{-1}=h',
\]
$x'g''\theta(x')^{-1}$ satisfies the third condition in the assertion.
This element is as desired.
\end{proof}

\subsection{Transfer factors}
As in the previous subsection, we put $\mathbf{G}=\GL_{2n}$ and $\mathbf{H}=\SO_{2n+1}$.
We fix the following $\theta$-stable Whittaker datum $(\mathbf{B}, \lambda)$ of $\mathbf{G}$: 
\begin{itemize}
 \item $\mathbf{B}$ is the subgroup of upper triangular matrices in $\mathbf{G}$, and 
 \item $\lambda$ is the character of the unipotent radical $U=\mathbf{U}(F)$ of $\mathbf{B}(F)$ defined by 
 \[
 \lambda(x)=\psi(x_{12}+\cdots+x_{2n-1, 2n}) \text{ for } x=(x_{ij}) \in U,
 \]
 where $\psi$ is the fixed nontrivial additive character of $F$.
\end{itemize}
Then we have the normalized absolute transfer factor $\Delta_{\mathbf{H},\mathbf{G}}$ for $\mathbf{G}$ and $\mathbf{H}$ with respect to $(\mathbf{B}, \lambda)$.
This is a function
\[
\Delta_{\mathbf{H},\mathbf{G}} \colon H^{\srs} \times G^{\strs} \ra \C,
\]
which has the following properties.
\begin{itemize}
 \item The value $\Delta_{\mathbf{H},\mathbf{G}}(h, g)$ is nonzero only if $h$ is a norm of $g$.
 \item If $h_1, h_2 \in H^{\srs}$ are stably conjugate, then $\Delta_{\mathbf{H},\mathbf{G}}(h_1, g)=\Delta_{\mathbf{H},\mathbf{G}}(h_2, g)$.
 \item If $g_1, g_2 \in G^{\strs}$ are $\theta$-conjugate, then $\Delta_{\mathbf{H},\mathbf{G}}(h, g_1)=\Delta_{\mathbf{H},\mathbf{G}}(h, g_2)$.
\end{itemize}

Our purpose in this subsection is to show the following proposition (the triviality of $\Delta_{\mathbf{H},\mathbf{G}}$):

\begin{prop}\label{prop:Delta}
Let $h \in H^{\srs}$ and $g \in G^{\strs}$.
If $h$ is a norm of $g$, then the transfer factor $\Delta_{\mathbf{H},\mathbf{G}}(h, g)$ is equal to $1$.
\end{prop}

The transfer factor $\Delta_{\mathbf{H},\mathbf{G}}$ is defined as the product of $\Delta_\mathrm{I}$, $\Delta_\mathrm{II}$, $\Delta_\mathrm{III}$ and $\Delta_\mathrm{IV}$ (since $\mathbf{G}$ and $\mathbf{H}$ are split, the contribution of the local $\varepsilon$-factor is trivial).
However, by using Waldspurger's formula for the transfer factors for classical groups (\cite[1.10 Proposition]{MR2672539}), 
we know that the product of the three factors $\Delta_{\mathrm{I}}$, $\Delta_{\mathrm{II}}$, and $\Delta_{\mathrm{III}}$ is trivial.
Indeed, an elliptic endoscopic group of the twisted $\GL_{2n}$ is in the form of
\[
\SO_{d^{+}}\times\SO_{d^{-}}, 
\]
where 
\begin{itemize}
 \item  $d^{+}\in\Z$ is a non-negative odd number and $d^{-}\in\Z$ is a non-negative even number such that $d^{+}+d^{-}=2n+1$, 
 \item $\SO_{d^{+}}$ is the split odd special orthogonal group, and
 \item $\SO_{d^{-}}$ is a quasi-split even special orthogonal group
\end{itemize}
and only $\SO_{d^{-}}$ contributes to the product $\Delta_{\mathrm{I}}\cdot\Delta_{\mathrm{II}}\cdot\Delta_{\mathrm{III}}$ nontrivially (see \cite[1.8 and 1.10]{MR2672539} for details).
In our setting, $d^{+}=2n+1$ and $d^{-}=0$.
Hence the product $\Delta_{\mathrm{I}}\cdot\Delta_{\mathrm{II}}\cdot\Delta_{\mathrm{III}}$ is trivial, and Proposition \ref{prop:Delta} is reduced to showing the triviality of $\Delta_{\mathrm{IV}}$.

We recall the definition of $\Delta_{\mathrm{IV}}$.
Let $D_{\mathbf{H}}$ be the Weyl discriminant: 
\[
D_{\mathbf{H}}(h) := \big|\det ( \Ad(h)-1 \mid \mathfrak{h}/{\mathfrak{t}_{\mathbf{H},h}})\big|^\frac{1}{2} \text{ for } h \in H^{\srs},
\]
where $\mathfrak{h}$ and $\mathfrak{t}_{\mathbf{H},h}$ are the Lie algebras of $\mathbf{H}$ and $\mathbf{T}_{\mathbf{H},h}:=\Cent_{\mathbf{H}}(h)$, respectively.
Let $D_{\mathbf{G},\theta}$ be the $\theta$-twisted Weyl discriminant: 
\[
D_{\mathbf{G},\theta}(g) := \big|\det \left( \Ad(g)\circ \theta-1 \mid \mathfrak{g}/{\mathfrak{t}_{g}}\right)\big|^\frac{1}{2} \text{ for } g \in G^{\strs},
\]
where $\mathfrak{g}$ and $\mathfrak{t}_{g}$ are the Lie algebras of $\mathbf{G}$ and $\mathbf{T}_{g}:=\Cent_{\mathbf{G}}(\Cent_{\mathbf{G}}(g, \theta)^0)$, respectively.
Here $\mathbf{T}_{g}$ is a maximal torus of $\mathbf{G}$.
Indeed, since $g$ is $\theta$-semisimple, we can find $x\in\mathbf{G}(\ol{F})$ such that $xg\theta(x)^{-1}\in\mathbf{T}(\ol{F})$.
As we have
\begin{align*}
\mathbf{T}_{xg\theta(x)^{-1}}
&=\Cent_{\mathbf{G}}\bigl(\Cent_{\mathbf{G}}(xg\theta(x)^{-1}, \theta)^0\bigr)\\
&=\Cent_{\mathbf{G}}\bigr(x\Cent_{\mathbf{G}}(g, \theta)^0x^{-1}\bigl)\\
&=x\Cent_{\mathbf{G}}\bigl(\Cent_{\mathbf{G}}(g, \theta)^0\bigr)x^{-1}=x\mathbf{T}_{g}x^{-1},
\end{align*}
it suffices to show that $\mathbf{T}_{xg\theta(x)^{-1}}$ is a maximal torus.
This follows from the following lemma.

\begin{lem}\label{lem:centcent}
If a $\theta$-regular $\theta$-semisimple element $g$ of $\mathbf{G}(\ol{F})$ belongs to the diagonal maximal torus $\mathbf{T}$ of $\mathbf{G}$, then we have $\mathbf{T}_{g}=\mathbf{T}$.
\end{lem}

\begin{proof}
As $g$ belongs to $\mathbf{T}$, every element $x$ of the $\theta$-fixed part $\mathbf{T}^{\theta}$ of $\mathbf{T}$ satisfies
\[
xg\theta(x)^{-1}=xgx^{-1}=g.
\]
Thus the $\theta$-twisted centralizer $\Cent_{\mathbf{G}}(g,\theta)$ of $g$ contains $\mathbf{T}^{\theta}$.
Since the action of $\theta$ on $\mathbf{T}$ is given by
\[
\diag(t_{1},\ldots, t_{2n}) \mapsto \diag\bigl(t_{2n}^{-1},\ldots, t_{1}^{-1}\bigr),
\]
$\mathbf{T}^{\theta}$ is a torus of rank $n$ given by
\[
\{\diag(t_{1},\ldots,t_{2n})\in\mathbf{T}\mid t_{1}=t_{2n}^{-1},\ldots,t_{n}=t_{n+1}^{-1}\}.
\]
In particular $\mathbf{T}^{\theta}$ is connected, hence $\Cent_{\mathbf{G}}(g,\theta)^{0}$ contains $\mathbf{T}^{\theta}$.
Therefore we have 
\[
\Cent_{\mathbf{G}}\bigl(\Cent_{\mathbf{G}}(g,\theta)^{0}\bigr)\subseteq\Cent_{\mathbf{G}}\bigl(\mathbf{T}^{\theta}\bigr).
\]

Since $\mathbf{T}^{\theta}$ contains a regular semisimple element of $\mathbf{G}$, $\Cent_{\mathbf{G}}(\mathbf{T}^{\theta})$ is equal to $\mathbf{T}$.
On the other hand, as $g$ is $\theta$-regular semisimple, $\Cent_{\mathbf{G}}(g,\theta)^{0}$ is a torus of $\mathbf{G}$.
In particular, $\Cent_{\mathbf{G}}(\Cent_{\mathbf{G}}(g,\theta)^{0})$ contains a maximal torus of $\mathbf{G}$.
Therefore we get
\[
\Cent_{\mathbf{G}}\bigl(\Cent_{\mathbf{G}}(g,\theta)^{0}\bigr)=\Cent_{\mathbf{G}}\bigl(\mathbf{T}^{\theta}\bigr)=\mathbf{T}.
\]
\end{proof}

We take $g\in G^{\strs}$, and let $h\in H^{\srs}$ be a norm of $g$.
Then the fourth factor is defined by 
\[
\Delta_\mathrm{IV}(h, g) = \frac{D_{\mathbf{G},\theta}(g)}{D_{\mathbf{H}}(h)}.
\]

\begin{lem}\label{lem:Delta4}
Let $h \in H^{\srs}$ and $g \in G^{\strs}$.
If $h$ is a norm of $g$, then we have
\[
D_{\mathbf{G},\theta}(g)=D_{\mathbf{H}}(h).
\]
In particular, we have $\Delta_{\mathrm{IV}}(h,g)=1$.
\end{lem}

\begin{proof}
First, we compute $D_{\mathbf{G},\theta}(g)$.
As $g$ is $\theta$-semisimple, we can find $x\in\mathbf{G}(\ol{F})$ such that $xg\theta(x)^{-1}\in\mathbf{T}(\ol{F})$.
Since the $\theta$-twisted Weyl discriminant is invariant under $\theta$-conjugation, it suffices to compute $D_{\mathbf{G},\theta}(xg\theta(x)^{-1})$.
We put 
\[
t:=xg\theta(x)^{-1}=\diag(t_{1},\ldots,t_{2n}).
\]
Then we have $\mathbf{T}_{t}=\mathbf{T}$ by Lemma \ref{lem:centcent}.

Now we compute $D_{\mathbf{G},\theta}(t) := |\det \left( \Ad(t)\circ \theta-1 \mid \mathfrak{g}/{\mathfrak{t}}\right)|^\frac{1}{2}$
, where $\mathfrak{t}$ is the Lie algebra of $\mathbf{T}$.
We write $X_{i,j}$ for the root vector of the root $e_{i}-e_{j}$ whose $(i,j)$-th entry is given by $1$ and all of the other entries are given by $0$.
Then we have
\begin{align*}
\theta(X_{i,j})
&=-J\,{}^{t}\!X_{i,j}J^{-1}\\
&=(-1)^{i+j+1}X_{2n+1-j,2n+1-i}.
\end{align*}
Thus $X_{i,j}$ is an eigenvector of $\theta$ if and only if $i+j=2n+1$.
In this case we have
\[
\bigl(\Ad(t)\circ \theta-1\bigr)(X_{i,2n+1-i})=(t_{i}/t_{2n+1-i}-1)X_{i,2n+1-i}.
\]
On the other hand, for $i+j\neq2n+1$, the $\ol{F}$-vector space spanned by $X_{i,j}$ and $X_{2n+1-j,2n+1-i}$ is stable under $\Ad(t)\circ \theta-1$, and the representation matrix of $\Ad(t)\circ \theta-1$ with respect to $X_{i,j}$ and $X_{2n+1-j,2n+1-i}$ is given by
\[
\begin{pmatrix}
-1&(-1)^{(2n+1-j)+(2n+1-i)+1}{t_{i}}/{t_{j}}\\
(-1)^{i+j+1}{t_{2n+1-j}}/{t_{2n+1-i}}&-1
\end{pmatrix}.
\]

Therefore if we set 
\[
\tilde{t}_{i}:=\frac{t_{i}}{t_{2n+1-i}},
\]
then $D_{\mathbf{G},\theta}(g)=|\det \left( \Ad(t)\circ \theta-1 \mid \mathfrak{g}/{\mathfrak{t}}\right)|^\frac{1}{2}$ is given by the product of
\[
D_{\mathbf{G},\theta}(g)_{1}:=\prod_{1\leq i \leq 2n} \big|\tilde{t}_{i}-1\big|^{\frac{1}{2}}
\]
and
\[
D_{\mathbf{G},\theta}(g)_{2}:=\prod_{\begin{subarray}{c} 1\leq i\neq j \leq 2n \\ i>2n+1-j\end{subarray}} \big|\tilde{t}_{i}\tilde{t}_{2n+1-j}-1\big|^{\frac{1}{2}}.
\]
Note that we have
\[
\bigl(D_{\mathbf{G},\theta}(g)_{2}\bigr)^{2}
=\prod_{\begin{subarray}{c} 1\leq i\neq j \leq 2n \\ i\neq2n+1-j\end{subarray}} \big|\tilde{t}_{i}\tilde{t}_{2n+1-j}-1\big|^{\frac{1}{2}}
=\prod_{\begin{subarray}{c} 1\leq i\neq j \leq 2n \\ i\neq2n+1-j\end{subarray}} \big|\tilde{t}_{i}\tilde{t}_{j}-1\big|^{\frac{1}{2}}.
\]


We next compute $D_{\mathbf{H}}(h)$.
Let $s =\diag(s_1, \ldots, s_n, 1, s_n^{-1}, \ldots, s_1^{-1}) \in \mathbf{T}_{\mathbf{H}}(\ol{F})$ be an element which is conjugate to $h$ in $\mathbf{H}(\ol{F})$.
By a similar, but simpler, argument to that above, $D_{\mathbf{H}}(h)=|\det\left(\Ad(s)-1\mid\mathfrak{h}/{\mathfrak{t}_\mathrm{H}}\right)|^{\frac{1}{2}}$ is given by the product of
\[
D_{\mathbf{H}}(h)_{1}:=\prod_{i=1}^{n} |s_{i}-1|^{\frac{1}{2}}\cdot\bigl|s_{i}^{-1}-1\bigr|^{\frac{1}{2}}
\]
and
\[
D_{\mathbf{H}}(h)_{2}:=\prod_{1\leq i< j \leq n} |s_{i}s_{j}-1|^{\frac{1}{2}} \cdot \biggl|\frac{s_{i}}{s_{j}}-1\biggr|^{\frac{1}{2}} \cdot \biggl|\frac{s_{j}}{s_{i}}-1\biggr|^{\frac{1}{2}} \cdot \biggl|\frac{1}{s_{i}s_{j}}-1\biggr|^{\frac{1}{2}}.
\]
Note that 
\[
\bigl(D_{\mathbf{H}}(h)_{2}\bigr)^{2}=\prod_{1\leq i\neq j \leq n} |s_{i}s_{j}-1|^{\frac{1}{2}} \cdot \biggl|\frac{s_{i}}{s_{j}}-1\biggr|^{\frac{1}{2}} \cdot \biggl|\frac{s_{j}}{s_{i}}-1\biggr|^{\frac{1}{2}} \cdot \biggl|\frac{1}{s_{i}s_{j}}-1\biggr|^{\frac{1}{2}}.
\]

On the other hand, by Lemma \ref{lem:eig}, the multi-set $\{s_{1}^{\pm1}, \ldots, s_{n}^{\pm1}\}$ is equal to the multi-set of eigenvalues $\mathrm{Eig}(\mathcal{N}(g))$ of $\mathcal{N}(g)$.
Since we have $\mathcal{N}(t)=x\mathcal{N}(g)x^{-1}$, $\mathrm{Eig}(\mathcal{N}(g))$ is equal to $\mathrm{Eig}(\mathcal{N}(t))=\{\tilde{t}_{1}, \ldots, \tilde{t}_{2n}\}$.
Hence we have
\[
\bigl\{s_{1}^{\pm1}, \ldots, s_{n}^{\pm1}\bigr\}=\bigl\{\tilde{t}_{1}, \ldots, \tilde{t}_{2n}\bigr\}.
\]
From this, we get $D_{\mathbf{G},\theta}(g)_{1}=D_{\mathbf{H}}(h)_{1}$.

We finally compare $D_{\mathbf{G},\theta}(g)_{2}$ with $D_{\mathbf{H}}(h)_{2}$.
By the assumption of the regularity of $h$, we have $s_{i}\neq s_{j}^{\pm1}$ for every $j\neq i$.
Moreover, by the definition of $\tilde{t}_{i}$, we have
\[
\tilde{t}_{i}^{-1}=\tilde{t}_{2n+1-i}.
\]
Thus, by permuting the indices of the $s_{i}$ appropriately, we may assume that 
\begin{align*}
s_{1}&= \tilde{t}_{1},\, s_{1}^{-1}=\tilde{t}_{2n}, \\
&\vdotswithin{=} \\
s_{n}&= \tilde{t}_{n},\, s_{1}^{-1}=\tilde{t}_{n+1}. 
\end{align*}
Then we have
\begin{align*}
&\prod_{1\leq i\neq j \leq n} |s_{i}s_{j}-1|^{\frac{1}{2}} \cdot \biggl|\frac{s_{i}}{s_{j}}-1\biggr|^{\frac{1}{2}} \cdot \biggl|\frac{s_{j}}{s_{i}}-1\biggr|^{\frac{1}{2}} \cdot \biggl|\frac{1}{s_{i}s_{j}}-1\biggr|^{\frac{1}{2}}\\
&=\prod_{1\leq i\neq j \leq n} \big|\tilde{t}_{i}\tilde{t}_{j}-1\big|^{\frac{1}{2}} \cdot \big|\tilde{t}_{i}\tilde{t}_{2n+1-j}-1\big|^{\frac{1}{2}} \cdot \big|\tilde{t}_{j}{\tilde{t}_{2n+1-i}}-1\big|^{\frac{1}{2}} \cdot \big|\tilde{t}_{2n+1-i}\tilde{t}_{2n+1-j}-1\big|^{\frac{1}{2}}\\
&=\prod_{\begin{subarray}{c} 1\leq i\neq j \leq 2n \\ j\neq2n+1-i\end{subarray}} \big|\tilde{t}_{i}\tilde{t}_{j}-1\big|^{\frac{1}{2}}.
\end{align*}
Thus we get $(D_{\mathbf{G},\theta}(g)_{2})^{2}=(D_{\mathbf{H}}(h)_{2})^{2}$, hence $D_{\mathbf{G},\theta}(g)_{2}=D_{\mathbf{H}}(h)_{2}$.
\end{proof}

\subsection{Norms of $1+\varphi_u$ and $\varphi_u(1+\varphi_u)$}
We first prove a lemma about $\theta$-semisimplicity of semisimple elements in $\mathbf{G}$.

\begin{lem}\label{lem:tss}
Let $g \in \mathbf{G}(\ol{F})$ be a semisimple element such that $\mathcal{N}(g)=g\theta(g) \in \mathbf{G}(\ol{F})$ is strongly regular semisimple.
If $g\theta(g)=\theta(g)g$, then $g$ is $\theta$-semisimple.
\end{lem}

\begin{proof}
Since $F$ has characteristic zero, $g$ is $\theta$-semisimple if and only if $\mathrm{Int}(g) \circ \theta$ preserves a pair $(\mathbf{B}', \mathbf{T}')$ in $\mathbf{G}$ (see Section 1.1 in \cite{MR1687096}).
If there exists $x \in \mathbf{G}(\ol{F})$ such that $xg\theta(x)^{-1}$ belongs to the diagonal torus $\mathbf{T}(\ol{F})$, then the pair $(x^{-1}\mathbf{B}x, x^{-1}\mathbf{T}x)$ is $\mathrm{Int}(g) \circ \theta$-invariant.
Therefore it suffices to prove the existence of such $x \in \mathbf{G}(\ol{F})$.

As $g$ is semisimple, $\theta(g)$ is also semisimple.
Since $g\theta(g)=\theta(g)g$, we can take $x \in \mathbf{G}(\ol{F})$ such that $xgx^{-1} \in \mathbf{T}(\ol{F})$ and $x\theta(g)x^{-1} \in \mathbf{T}(\ol{F})$.
Note that this element also satisfies $x\mathcal{N}(g)x^{-1} \in \mathbf{T}(\ol{F})$.
On the other hand, since the element $\mathcal{N}(g)$ is $\theta$-invariant, $J\,{}^{t}\!\mathcal{N}(g)^{-1}J^{-1}$ equals $\mathcal{N}(g)$.
In particular, if we let $\{t_{1}, \ldots, t_{2n}\}$ be the multi-set $\mathrm{Eig}(\mathcal{N}(g))$ of eigenvalues of $\mathcal{N}(g)$, then we have
\[
\{t_{1}, \ldots, t_{2n}\}=\{t_{1}^{-1}, \ldots, t_{2n}^{-1}\}.
\]
By combining with the assumption that $\mathcal{N}(g)$ is strongly regular semisimple (i.e., $\mathrm{Eig}(\mathcal{N}(g))$ is multiplicity-free), we have
\[
\mathrm{Eig}\bigl(\mathcal{N}(g)\bigr)=\{t_{1}, \ldots, t_{n}, t_{n}^{-1}, \ldots, t_{1}^{-1}\}.
\]
Therefore, by replacing $x$ with $wx$ for an appropriate element $w$ of the Weyl group $\Omega_{\mathbf{T}}$, we may assume that 
\[
x\mathcal{N}(g)x^{-1} = \diag\bigl(t_{1}, \ldots, t_{n}, t_{n}^{-1}, \ldots, t_{1}^{-1}\bigr) \in \mathbf{T}^\theta(\ol{F}) \subseteq \mathbf{T}(\ol{F}).
\]

Then we have 
\[
x\mathcal{N}(g)x^{-1}=\theta\bigl(x\mathcal{N}(g)x^{-1}\bigr)=\theta(x)\mathcal{N}(g)\theta(x)^{-1},
\]
hence $x^{-1}\theta(x)$ belongs to $\Cent_{\mathbf{G}}(\mathcal{N}(g))$.
Since $\mathcal{N}(g)$ belongs to $x^{-1}\mathbf{T}(\ol{F})x$ and $\mathcal{N}(g)$ is strongly regular semisimple, we have $\Cent_{\mathbf{G}}(\mathcal{N}(g))=x^{-1}\mathbf{T}x$.
Therefore $x\theta(x)^{-1} \in \mathbf{T}(\ol{F})$, and 
$xg\theta(x)^{-1}=xgx^{-1}\cdot x\theta(x)^{-1} \in \mathbf{T}(\ol{F})$.
This element $x$ is as desired.
\end{proof}

For $u \in k^{\times}$, we consider the following element:
\[
\mathfrak{N}(1+\varphi_u) := 
(1-\varpi u)^{-1} 
\begin{pmatrix}
 1&&&&\\
 \varpi u&1+\varpi u&&\mbox{\Large 2}&\\
 \vdots&&\ddots&&\\
 \varpi u&\mbox{\Large $2\varpi u$}&&1+\varpi u&\\
 (\varpi u)^2/2&\varpi u&\hdots&\varpi u&1
\end{pmatrix}\\
\in I_H^+ \subseteq H.
\]
Here, the big $2$ means that all upper triangular entries are given by 2, and the big $2\varpi u$ means that all lower triangular entries which lie in neither the first column nor the last row are given by $2\varpi u$.

\begin{prop}\label{prop:N}
\begin{enumerate}
\item
The element $1+\varphi_u \in G$ is strongly $\theta$-regular $\theta$-semisimple and $\mathfrak{N}(1+\varphi_u) \in H$ is strongly regular semisimple.
Moreover $\mathfrak{N}(1+\varphi_u)$ is a norm of $1+\varphi_u$.
\item
The element $\varphi_u(1+\varphi_u) \in G$ is strongly $\theta$-regular $\theta$-semisimple and $\varphi'_{u/2}\mathfrak{N}(1+\varphi_u) \in H$ is strongly regular semisimple.
Moreover $\varphi'_{u/2}\mathfrak{N}(1+\varphi_u)$ is a norm of $\varphi_u(1+\varphi_u)$.
\end{enumerate}
\end{prop}

\begin{proof}
Let us prove (1).
We first show that $1+\varphi_u$ is $\theta$-semisimple.
By Lemma \ref{lem:Eisen}, $1+\varphi_u$ is semisimple.
Since 
\begin{align*}
\theta(1+\varphi_u)
&=J(1+{}^t\varphi_u)^{-1}J^{-1}\\
&=J(1-{}^t\varphi_u+{}^t\varphi_u^2-\cdots)J^{-1}\\
&=1+\varphi_u+\varphi_u^2+\cdots,
\end{align*}
$1+\varphi_u$ commutes with $\theta(1+\varphi_u)$.
Moreover we can check that  
\[
\mathcal{N}(1+\varphi_u)= 
(1+\varphi_u)\theta(1+\varphi_u)= 
(1-\varpi u)^{-1}
\begin{pmatrix}
 1+\varpi u&&\mbox{\Large 2}\\
 &\ddots&\\
 \mbox{\Large $2\varpi u$}&&1+\varpi u
\end{pmatrix}.
\]
Since the characteristic of $k$ is not equal to $2$, this is an affine generic element of $G$.
Hence this is strongly regular semisimple by Lemma \ref{lem:Eisen}.
Therefore $1+\varphi_u$ is $\theta$-semisimple by Lemma \ref{lem:tss}.

We next show $\mathfrak{N}(1+\varphi_u)$ is semisimple and corresponds to $1+\varphi_u$ via $\mathcal{A}_{\mathbf{H}/\mathbf{G}}$.
Let $X$ be the matrix
\[
\begin{pmatrix}
 1&&\\
 &\ddots&\\
 -\varpi u/2&&1
\end{pmatrix} \in \GL_{2n+1}(F),
\] then we have 
\[
X^{-1}\mathfrak{N}(1+\varphi_u)X =
(1-\varpi u)^{-1}
\begin{pmatrix}
 1-\varpi u&&&\\
 0&1+\varpi u&&\mbox{\Large 2}\\
 \vdots&&\ddots&\\
 0&\mbox{\Large $2\varpi u$}&&1+\varpi u
\end{pmatrix}.
\]
By the affine genericity of $\mathcal{N}(1+\varphi_u)$, this element is semisimple.
Moreover $\mathfrak{N}(1+\varphi_u)$ corresponds to $1+\varphi_u$ via $\mathcal{A}_{\mathbf{H}/\mathbf{G}}$ by Lemma \ref{lem:norm1}.

Finally, by Lemma \ref{lem:norm0}, $1+\varphi_u$ is strongly $\theta$-regular and $\mathfrak{N}(1+\varphi_u)$ is strongly regular.

For (2), by using the following equalities, we can apply the same argument as in (1):
\begin{align*}
\theta\bigl(\varphi_u(1+\varphi_u)\bigr)
&=\theta(\varphi_u)\theta(1+\varphi_u)\\
&=-\varphi_u^{-1}(1+\varphi_u+\varphi_u^2+\cdots),
\end{align*}
\[
\mathcal{N}\bigl(\varphi_u(1+\varphi_u)\bigr)
= \varphi_u(1+\varphi_u)\theta\bigl(\varphi_u(1+\varphi_u)\bigr) 
=-\mathcal{N}(1+\varphi_u), \text{ and}
\]
\[
X^{-1}\varphi'_{u/2}\mathfrak{N}(1+\varphi_u)X =
(-1+\varpi u)^{-1}
\begin{pmatrix}
 -1+\varpi u&2&\hdots&2/(\varpi u)\\
 0&1+\varpi u&&\mbox{\Large 2}\\
 \vdots&&\ddots&\\
 0&\mbox{\Large $2\varpi u$}&&1+\varpi u
\end{pmatrix}.
\]
\end{proof}

\section{Main theorem}\label{sec:main}
\subsection{Endoscopic character relation}
In this subsection we recall the endoscopic classification of representations of classical groups in \cite{MR3135650}, for the cases of odd special orthogonal groups.

As in previous sections, let $\mathbf{G}$ be $\GL_{2n}$ over $F$, and $\theta$ the automorphism of $\mathbf{G}$ over $F$ defined by
\[
g \mapsto J\,{}^{t}\!g^{-1}J^{-1} \text{, where }
J=\begin{pmatrix}
 &&&1\\
 &&-1&\\
 &\adots&&\\
 (-1)^{2n-1}&&&
\end{pmatrix}.
\]
Let $\mathbf{H}$ be $\SO_{2n+1}$ over $F$.

We write $\Phi(G)$ (resp.\ $\Phi(H)$) for the set of $\widehat{G}$-conjugacy classes of $L$-parameters of $G$ (resp.\ $\widehat{H}$-conjugacy classes of $L$-parameters of $H$).
For $\phi_{H} \in \Phi(H)$, we set
\begin{align*}
 S_{\phi_{H}} &:= \Cent_{\widehat{H}}\bigl(\mathrm{Im}(\phi_H)\bigr) \text{, and }\\
 \mathcal{S}_{\phi_{H}} &:= S_{\phi_{H}} / S_{\phi_{H}}^0Z(\widehat{H}).
\end{align*}
Here, we fix a representative of $\phi_{H}$ and denote it again by $\phi_{H}$.

Before we state Arthur's theorem, we explain the normalizations of $\theta$-twisted distribution characters of self-dual irreducible smooth representations of $G$ determined by the Whittaker datum.
In the following, we fix the $\theta$-stable Whittaker datum $(\mathbf{B}, \lambda)$ defined in Section \ref{sec:norm}.3.
Let $\pi$ be a self-dual irreducible tempered representation of $G$.
Then $\pi$ is generic for $(\mathbf{B}, \lambda)$ (hence for every Whittaker datum), 
that is 
\[
\Hom_{U}(\pi, \lambda)\neq0.
\]
In general, it is known that the dimension of $\Hom_{U}(\pi, \lambda)$ is not greater than one.
Therefore $\Hom_{U}(\pi, \lambda)$ is $1$-dimensional if $\pi$ is generic.
If we take a nonzero element $\Lambda$ of $\Hom_{U}(\pi, \lambda)$, then it also belongs to $\Hom_{U}(\pi^{\theta}, \lambda^{\theta})=\Hom_{U}(\pi^{\theta}, \lambda)$.
On the other hand, since $\pi$ is self-dual, there exists a $G$-equivariant isomorphism $A\colon\pi\cong\pi^{\theta}$ (this is unique up to a scalar multiple).
Then the composite $\Lambda\circ A$  is again a nonzero element of $\Hom_{U}(\pi, \lambda)$, hence equals $c\cdot \Lambda$ for a scalar $c\in\C^{\times}$, which does not depend on $\Lambda$.

In the following, for a self-dual irreducible tempered representation $\pi$ of $G$, we take the unique intertwiner $A\colon\pi\cong\pi^{\theta}$ such that $\Lambda\circ A$ is equal to $\Lambda$.
We note that, for a simple supercuspidal representation of $G$, this normalization coincides with that of Section \ref{sec:char}.2.
Namely the following proposition holds.

\begin{prop}
Let $a \in k^{\times}$ and $\zeta\in\{\pm1\}$, and we consider the simple supercuspidal representation $\pi_{a,\zeta}$.
The intertwiner $A\colon\pi_{a,\zeta}\cong\pi_{a,\zeta}^{\theta}$ induced from the identity map $\chi_{a,\zeta}\ra\chi_{a,\zeta}$ $(=\chi_{a,\zeta}^{\theta})$ satisfies $\Lambda=\Lambda\circ A$ for every $\Lambda\in\Hom_{U}(\pi_{a,\zeta}, \lambda)$.
\end{prop}

\begin{proof}
Since the ratio of $\Lambda$ to $\Lambda\circ A$ does not depend on $\Lambda$, it suffices to find one homomorphism $\Lambda\in\Hom_{U}(\pi_{a,\zeta}, \lambda)$ which is nonzero and satisfies $\Lambda=\Lambda\circ A$.

We first note that the intertwiner 
\[
A\colon \pi_{a,\zeta} = \cInd_{Z I^{+}\lan\varphi_{a^{-1}}\ran}^{G} \chi_{a,\zeta} = \cInd_{Z I^{+}\lan\varphi_{a^{-1}}\ran}^{G} \chi_{a,\zeta}^{\theta} \cong \pi_{a,\zeta}^{\theta}
\] 
induced from the identity map of $\chi_{a,\zeta}$ is given by
\[
 f \mapsto \theta^{\ast}(f)=f\circ\theta.
\]
Here we regard $f\in \pi_{a,\zeta}$ as a function on $G$.
Since $\Lambda$ is a scalar multiple of $\Lambda\circ A$, it is enough to find $\Lambda\in\Hom_{U}(\pi_{a,\zeta}, \lambda)$ satisfying
\[
\Lambda(\tilde{\chi}_{a,\zeta})=(\Lambda\circ A)(\tilde{\chi}_{a,\zeta})\neq0,
\]
where $\tilde{\chi}_{a,\zeta}\in\pi_{a,\zeta}$ is the zero extension of $\chi_{a,\zeta}$ to $G$:
\[
\tilde{\chi}_{a,\zeta}(g) := 
\begin{cases}
 \chi_{a,\zeta}(g) & (g\in Z I^{+}\lan\varphi_{a^{-1}}\ran),\\
 0 & (\text{otherwise}).
\end{cases}
\]
However, for every $\Lambda\in\Hom_{U}(\pi_{a,\zeta}, \lambda)$, we have
\[
(\Lambda\circ A)(\tilde{\chi}_{a,\zeta})
=\Lambda(\tilde{\chi}_{a,\zeta}\circ\theta)
=\Lambda(\tilde{\chi}_{a,\zeta}).
\]
Thus we are reduced to finding $\Lambda\in\Hom_{U}(\pi_{a,\zeta}, \lambda)$ such that $\Lambda(\tilde{\chi}_{a,\zeta})\neq0$.

By the Frobenius reciprocities for the compact and smooth inductions, we have
\[
\Hom_{Z I^{+}\lan\varphi_{a^{-1}}\ran}(\chi_{a,\zeta}, \Ind_{U}^{G}\lambda)
\cong
\Hom_{G}\bigl(\pi_{a,\zeta}, \Ind_{U}^{G}\lambda\bigr)
\cong
\Hom_{U}(\pi_{a,\zeta}, \lambda).
\]
If we write $\Lambda'$ for the element of $\Hom_{Z I^{+}\lan\varphi_{a^{-1}}\ran}(\chi_{a,\zeta}, \Ind_{U}^{G}\lambda)$ corresponding to $\Lambda$ under these isomorphisms,
then, by chasing the constructions of the isomorphisms of the Frobenius reciprocities, we have 
\[
\Lambda(\tilde{\chi}_{a,\zeta})\neq0\Leftrightarrow\bigl(\Lambda'(1)\bigr)(1)\neq0.
\]
Thus it suffices to construct $\Lambda'\in\Hom_{Z I^{+}\lan\varphi_{a^{-1}}\ran}(\chi_{a,\zeta}, \Ind_{U}^{G}\lambda)$ satisfying $(\Lambda'(1))(1)\neq0$.

We put
\[
W(g) := 
\begin{cases}
 \lambda(u)\chi_{a,\zeta}(x) & (g=ux \text{ for } u\in U, x\in ZI^{+}\lan\varphi_{a^{-1}}\ran),\\
 0 & (\text{otherwise}).
\end{cases}
\]
This is well-defined since $\lambda$ coincides with $\chi_{a,\zeta}$ on $U\cap ZI^{+}\lan\varphi_{a^{-1}}\ran$, and an element of $\Ind_{U}^{G}\lambda$.
We define a $\C$-linear map $\Lambda'$ from $\chi_{a,\zeta}$ to $\Ind_{U}^{G}\lambda$ by $\Lambda'(1):=W$.
Then $\Lambda'$ is $Z I^{+}\lan\varphi_{a^{-1}}\ran$-equivariant and satisfies $(\Lambda'(1))(1)\neq0$.
Thus $\Lambda'$ is as desired.
\end{proof}

Now we state Arthur's local classification theorem.

\begin{thm}[{\cite[Theorems 1.5.1 and 2.2.1]{MR3135650}}]\label{thm:Arthur}
For every tempered $L$-parameter $\phi_{H} \in \Phi(H)$, there is a finite set $\Pi_{\phi_{H}}$ (which is called an $L$-packet) consisting of irreducible tempered representations of $H$, and the following properties hold.
\begin{itemize}
 \item There is a bijection from $\Pi_{\phi_{H}}$ to the group $\widehat{\mathcal{S}}_{\phi_{H}}$ of characters of $\mathcal{S}_{\phi_{H}}$.
 \item The distribution $\sum_{\pi_{H} \in \Pi_{\phi_{H}}} \tr \pi_H$ on $\mathcal{H}(H)$ is stable.
 \item For every $f \in \mathcal{H}(G)$, we have the following equality:
 \[
 \tr \pi_\theta(f) = \sum_{\pi_H \in \Pi_{\phi_{H}}} \tr \pi_H(f^H), 
 \]
 where $\pi$ is the self-dual irreducible tempered representation of $G$ corresponding to $\xi\circ\phi_{H}\in\Phi(G)$ via the local Langlands correspondence for $G$, $\tr \pi_\theta$ is its $\theta$-twisted distribution character with respect to the normalization determined by the fixed Whittaker datum $(\mathbf{B},\lambda)$, and $f^H\in\mathcal{H}(H)$ is a transfer of $f$ to $H$.
\end{itemize}
\end{thm}

We remark that in \cite{MR3135650} this theorem is stated in terms of the twisted space (for example, the $\theta$-twisted distribution of $\pi$ is regarded as a $\C$-linear functional on $\mathcal{H}(G\rtimes\theta)$).
We can translate this formalism into our language (see e.g.\ \cite[I.2.6 and I.5.11]{MR3632513}).

By using a stable version of the twisted Weyl integration formula, we can rewrite the relation in Theorem \ref{thm:Arthur} in terms of the characters $\Theta_{\pi, \theta}$ and $\Theta_{\pi_H}$ as follows (see e.g.\ Section 5 in \cite{MR3067291}).

\begin{thm}\label{thm:ECR}
Let $\phi_{H} \in \Phi(H)$, $\Pi_{\phi_{H}}$ the $L$-packet of $\phi_{H}$ defined in Theorem \ref{thm:Arthur}, and $\pi$ the irreducible smooth representation of $G$ which corresponds to $\xi\circ\phi_{H}\in\Phi(G)$ under the local Langlands correspondence for $G$.
Then, for every $g \in G^{\strs}$, we have the following equality: 
\[
\Theta_{\pi, \theta}(g) = \sum_{h\mapsto g} \frac{D_{\mathbf{H}}(h)^2}{D_{\mathbf{G},\theta}(g)^2} \Delta_{\mathbf{H},\mathbf{G}}(h, g) \sum_{\pi_{H}\in\Pi_{\phi_{H}}}\Theta_{\pi_H}(h),
\]
where the sum is over stable conjugacy classes of norms $h \in H^{\srs}$ of $g$.
\end{thm}

We next specialize this relation to the case of simple supercuspidal representations.
We use the following result on the parity of simple supercuspidal representations. 

\begin{prop}[{\cite[Corollary 4.6 (iii)]{Mieda:2016}}]\label{prop:Mieda}
Let $\pi$ be a self-dual simple supercuspidal representation of $G$ with trivial central character.
Then, up to $\widehat{G}$-conjugation, the $L$-parameter of $\pi$ factors through the $L$-group ${}^{L}\mathbf{H}$ of $\mathbf{H}$.
\end{prop}

Let $\pi$ be a self-dual simple supercuspidal representation of $G$ with trivial central character, and $\phi$ the $L$-parameter of $\pi$.
Then, by Proposition \ref{prop:Mieda}, $\phi$ induces an $L$-parameter $\phi_{H}\in\Phi(H)$.
By Theorem \ref{thm:Arthur}, $\phi_{H}$ defines a finite set $\Pi_{\phi_{H}}$ consisting of irreducible tempered representations of $H$.
Since $\pi$ is supercuspidal, the corresponding $L$-parameter $\phi \colon W_F \ra \widehat{G}=\GL_{2n}(\C)$ is irreducible as a representation of $W_F$.
Therefore $\Cent_{\widehat{G}}(\mathrm{Im}(\phi))$ consists of scalar matrices, and so does $\Cent_{\widehat{H}}(\mathrm{Im}(\phi_H))$.
Hence the group $\mathcal{S}_{\phi_{H}}$ is trivial and $\Pi_{\phi_{H}}$ is a singleton by Theorem \ref{thm:Arthur}.
We denote by $\pi_H$ the unique representation in $\Pi_{\phi_{H}}$ and say that $\pi_H$ is $\mathit{associated}$ $\mathit{to}$ $\pi$.
We remark that the character $\Theta_{\pi_H}$ of $\pi_H$ is stable by Theorem \ref{thm:Arthur}.
\[
\xymatrix{
\pi & \ar@{<~>}[r]^-{\text{LLC for $\GL_{2n}$}} && W_F \ar[rd]_-{\phi_{H}} \ar[r]^-{\phi} & {}^{L}\mathbf{G}\\
{\Pi_{\phi_{H}}=\{\pi_{H}\}} & \ar@{<~>}[r]^-{\text{Theorem \ref{thm:Arthur}}} &&& {}^{L}\mathbf{H} \ar@{^{(}->}[u]_-{\xi} \\
}
\]

Our purpose is to determine $\pi_H$ by using the relation in Theorem \ref{thm:ECR}.
We can simplify this equality by Lemma \ref{lem:uniquenorm} and the computation of the transfer factor $\Delta_{\mathbf{H},\mathbf{G}}$ in Section 4.3.

\begin{cor}\label{cor:ECR}
Let $\Theta_{\pi, \theta}$ be the twisted character of $\pi$ with respect to the normalization determined by the fixed Whittaker datum $(\mathbf{B}, \lambda)$, and $\Theta_{\pi_H}$ the character of $\pi_H$.
Let $g \in G^{\strs}$ and $h \in H^{\srs}$ such that $h$ is a norm of $g$.
Then we have 
\[
\Theta_{\pi, \theta}(g)=\Theta_{\pi_H}(h).
\]
\end{cor}

\begin{proof}
Combining Theorem \ref{thm:ECR} with Lemma \ref{lem:uniquenorm}, we have
\[
\Theta_{\pi, \theta}(g)=\frac{D_{\mathbf{H}}(h)^2}{D_{\mathbf{G},\theta}(g)^2} \Delta_{\mathbf{H},\mathbf{G}}(h, g)\Theta_{\pi_H}(h).
\]
Moreover, by Proposition \ref{prop:Delta} and Lemma \ref{lem:Delta4}, we have
\[
\Theta_{\pi, \theta}(g)=\Theta_{\pi_H}(h).
\]
\end{proof}

In particular, for affine generic elements considered in Section \ref{sec:norm}.4, we get the following equalities of characters:
\begin{cor}\label{cor:CR}
Let $\zeta\in\{\pm1\}$ and $\pi_H$ the irreducible smooth representation of $H$ associated to the simple supercuspidal representation $\pi_{1,\zeta}$ of $G$.
\begin{enumerate}
\item
For $u \in k^{\times}$, we have 
\[
\Theta_{\pi_H}\bigl(\mathfrak{N}(1+\varphi_u)\bigr)=\Theta_{\pi_{1,\zeta}, \theta}(1+\varphi_u)=\Kl_{2^{2(n-1)}u}^{n+1}(\psi; 1, 2, \ldots, 2, 1).
\]
\item
For $u \in k^{\times}$, we have 
\[
\Theta_{\pi_H}\bigl(\varphi'_{u^2/2}\mathfrak{N}(1+\varphi_{u^2})\bigr)=\Theta_{\pi_{1,\zeta}, \theta}\bigl(\varphi_{u^2}(1+\varphi_{u^2})\bigr)=\zeta\left(\Kl_{2^nu}^n(\psi)+\Kl_{-2^nu}^n(\psi)\right).
\]
\end{enumerate}
\end{cor}

\begin{proof}
By Proposition \ref{prop:N}, $\mathfrak{N}(1+\varphi_u)\in H^{\srs}$ (resp.\ $\varphi_{u^2}(1+\varphi_{u^2})\in H^{\srs}$) is a norm of $1+\varphi_u\in G^{\strs}$ (resp.\ $\varphi_{u^2}(1+\varphi_{u^2}) \in G^{\strs}$).
Therefore we get the first equalities by Corollary \ref{cor:ECR}.

On the other hand, since $\mathcal{N}(1+\varphi_u)$ (resp.\ $-\mathcal{N}(\varphi_{u^2}(1+\varphi_{u^2}))$) is an affine generic element of $G$ (see the proof of Proposition \ref{prop:N}), $1+\varphi_{u}$ (resp.\ $1+\varphi_{u^{2}}$) satisfies the assumption of Proposition \ref{prop:charTGL} (resp.\ \ref{prop:charTGL'}).
The simple affine components of $1+\varphi_{u}$ and $1+\varphi_{u^{2}}$ are given by $(1, \ldots, 1, u)$ and $(1, \ldots, 1, u^{2})$, respectively.
Therefore, by Propositions \ref{prop:charTGL} and \ref{prop:charTGL'}, we have
\begin{align*}
\Theta_{\pi_{1,\zeta}, \theta}(1+\varphi_u)&=\Kl_{2^{2(n-1)}u}^{n+1}(\psi; 1, 2, \ldots, 2, 1), \text{ and}\\
\Theta_{\pi_{1,\zeta}, \theta}\bigl(\varphi_{u^2}(1+\varphi_{u^2})\bigr)&=\zeta\left(\Kl_{2^nu}^n(\psi)+\Kl_{-2^nu}^n(\psi)\right).
\end{align*}
Hence we get the second equalities.
\end{proof}

%
%
%

\subsection{Supercuspidality of $\pi_H$}
\begin{prop}\label{prop:supercuspidality}
Let $\pi$ be a self-dual simple supercuspidal representation of $G$ with trivial central character.
Let $\pi_H$ be the irreducible smooth representation of $H$ associated to $\pi$.
Then $\pi_H$ is supercuspidal.
\end{prop}

\begin{proof}
We use the theorem on a parametrization of supercuspidal representations of classical groups in \cite{MR3713922}.
When viewed as a representation of $W_F$, $\phi$ is irreducible by the supercuspidality of $\pi$.
Therefore the set $\mathrm{Jord}(\phi_H)$ is a singleton $\{(\pi, 1)\}$ (see Section 2 in \cite{MR3713922} for the definition of the set $\mathrm{Jord}(\phi_{H})$).
This satisfies the properties in \cite[Theorem 3.3]{MR3713922}, therefore $\pi_H$ is supercuspidal.
\end{proof}

\subsection{Existence of $I_H^{++}$-fixed vector}
Let $\zeta \in \{\pm1\}$ and $\pi_H$ the irreducible smooth representation of $H$ associated to $\pi_{1,\zeta}$.

\begin{lem}\label{lem:charlift}
Let $h \in H$ be an affine generic element with its simple affine components $(h_1, \ldots, h_n, h_{2n})$.
Then $\Theta_{\pi_H}(h)$ is equal to either $0$ or 
\[
\Kl_{h_1 h_2^2\cdots h_n^2 h_{2n}/2}^{n+1}(\psi; 1, 2, \ldots, 2, 1).
\]
\end{lem}

\begin{proof}
For such $h$, we take $g \in G$ satisfying the conditions in Proposition \ref{prop:affgennorm}.
Since $h \in H^{\srs}$, $g \in G^{\strs}$, and $h$ is a norm of $g$, 
we have 
\[
\Theta_{\pi_{1,\zeta}, \theta}(g)=\Theta_{\pi_H}(h)
\]
by Corollary \ref{cor:ECR}.
We compute the left-hand side of this equality.

If there is no $x \in G$ such that $xg\theta(x)^{-1} \in ZI^+\lan\varphi_{1}\ran$, 
then $\Theta_{\pi_{1,\zeta}, \theta}(g)=0$ by the twisted character formula (Theorem \ref{thm:TCF}).

Let us consider the case where there exists $x \in G$ such that $xg\theta(x)^{-1} \in ZI^+\lan\varphi_{1}\ran$.
By Lemma \ref{lem:sumGL}, we may assume $x \in T(q)$.
By replacing $g$ with $gz$ for some $z \in Z$, we may assume that $xg\theta(x)^{-1} \in I^+\lan\varphi_{1}\ran$.
Since $\theta(\varphi_{1})=-\varphi_{1}^{-1}$ and $\varphi_{1}$ normalizes $I^+$, 
we have 
\[
(-1)^k\varphi_{1}^k xg\theta(\varphi_{1}^k x)^{-1} \in I^+\sqcup I^+\varphi_{1} 
\]
for some $k \in \Z$.
By replacing $g$ with $(-1)^kg$ again, we may assume that 
\[
\varphi_{1}^k xg\theta(\varphi_{1}^k x)^{-1} \in I^+\sqcup I^+\varphi_{1}.
\]

If $\varphi_{1}^k xg\theta(\varphi_{1}^k x)^{-1}$ lies in $I^+\varphi_{1}$, then 
\[
\mathcal{N}\bigl(\varphi_{1}^kxg\theta(\varphi_{1}^kx)^{-1}\bigr) = \varphi_{1}^kx\mathcal{N}(g)(\varphi_{1}^kx)^{-1} \in \mathcal{N}(I^+\varphi_{1})\subseteq-I^+,
\] 
and this contradicts the assumption that $\mathcal{N}(g) \in I^+$.
Therefore $\varphi_{1}^k xg\theta(\varphi_{1}^k x)^{-1} \in I^+$.
Since the twisted  character is invariant under $\theta$-conjugacy, 
it suffices to compute 
\[
\Theta_{\pi_{1,\zeta}, \theta}\bigl(\varphi_{1}^kxg\theta(\varphi_{1}^kx)^{-1}\bigr).
\]

Let $x=\diag(t_1, \ldots, t_{2n})$. 
Let $(g_1, \ldots, g_{2n})$ be the simple affine components of  $\varphi_{1}^k xg\theta(\varphi_{1}^k x)^{-1}$.
Then the simple affine components of $\mathcal{N}(\varphi_{1}^kxg\theta(\varphi_{1}^kx)^{-1})$ are 
\[
(g_1+g_{2n-1}, \ldots, 2g_n, \ldots, g_{2n-1}+g_1, 2g_{2n}).
\]
On the other hand, since the simple affine components of $\mathcal{N}(g)$ are 
\[
(h_2, \ldots, h_n, h_n, \ldots, h_1, 2h_{2n}),
\]
those of $\mathcal{N}(\varphi_{1}^kxg\theta(\varphi_{1}^kx)^{-1})=\varphi_{1}^kx\mathcal{N}(g)(\varphi_{1}^kx)^{-1}$ are a cyclic permutation of
\[
\left(\frac{t_1}{t_2}h_2, \ldots, \frac{t_{2n-1}}{t_{2n}}h_1, 2\frac{t_{2n}}{t_1}h_{2n}\right).
\]

Therefore $\varphi_{1}^kxg\theta(\varphi_{1}^kx)^{-1}$ satisfies the assumption of Proposition \ref{prop:charTGL}, 
and we have
\begin{align*}
\Theta_{\pi_{1,\zeta}, \theta}\bigl(\varphi_{1}^kxg\theta(\varphi_{1}^kx)^{-1}\bigr) &= \Kl_{g_n (g_1 +g_{2n-1})^2 \cdots (g_{n-1}+g_{n+1})^2 g_{2n}}^{n+1} (\psi; 1, 2, \ldots, 2, 1)\\
&= \Kl_{h_1 h_2^2\cdots h_n^2 h_{2n}/2}^{n+1} (\psi; 1, 2, \ldots, 2, 1).
\end{align*}
\end{proof}

\begin{cor}\label{cor:depthbound}
The representation $(\pi_H, V_H)$ has a nonzero $I_H^{++}$-fixed vector.
\end{cor}

\begin{proof}
We take $u \in k^{\times}$ such that $\Kl_{2^{2(n-1)}u}^{n+1}(\psi; 1, 2, \ldots, 2, 1) \neq 0$ (such $u\in k^{\times}$ exists by Corollary \ref{cor:Kl}).
Let $h:=\mathfrak{N}(1+\varphi_u) \in I_H^+$ and $f := \bold{1}_{hI_H^{++}}$ the characteristic function of $hI_H^{++}$.
Since the subgroup $I_H^{++}$ is normal in $I_H^+$, we have $I_H^{++}hI_H^{++}=hI_H^{++}$.
Moreover $hI_{H}^{++}$ is contained in $H^{\rs}$ by Proposition \ref{prop:affgennorm}.
Therefore $f$ belongs to $\mathcal{H}(H, I_H^{++})$ and satisfies $\supp(f) \subseteq H^{\rs}$.

By a property of the character (Theorem \ref{thm:char}), we have
\[
\tr \pi_H(f) = \int_{H^{\rs}} f(y)\Theta_{\pi_H}(y) \,dy = \int_{hI_H^{++}} \Theta_{\pi_H}(y) \,dy.
\]
For $y \in hI_H^{++}$, $\Theta_{\pi_H}(y)$ is equal to either $\Kl_{2^{2(n-1)}u}^{n+1}(\psi; 1, 2, \ldots, 2, 1)$ or $0$, by Lemma \ref{lem:charlift}.
Moreover we have 
\[
\Theta_{\pi_H}(h) = \Theta_{\pi, \theta}(1+\varphi_u) = \Kl_{2^{2(n-1)}u}^{n+1}(\psi; 1, 2, \ldots, 2, 1) \neq 0
\]
by Corollary \ref{cor:CR}.
Since the character $\Theta_{\pi_H}$ is locally constant on $H^{\rs}$, we have
\[
\int_{hI_H^{++}} \Theta_{\pi_H}(y) \,dy \neq 0.
\]

On the other hand, by the definition of the distribution character, 
\[
\tr \pi_H(f) = \tr \Bigl(\pi_H(f) \,\Big\vert\, V_H^{I_H^{++}}\Bigr).
\]
Therefore we have $V_H^{I_H^{++}} \neq 0$.
\end{proof}

\subsection{Simple supercuspidality of $\pi_H$}
Let $\pi_{1,\zeta}$ and $\pi_H$ be as in the previous subsection.
In this subsection, we prove that $\pi_H$ is simple supercuspidal.

By Corollary \ref{cor:depthbound}, the finite abelian group $I_H^+/I_H^{++}$ acts on the nonzero subspace $\pi_H^{I_H^{++}}$ of $\pi_H$, and it decomposes into a direct sum of characters of $I_H^+/I_H^{++}$.
We take a character $\eta$ of $I_H^+$ which is contained in it.

\begin{lem}\label{lem:Mackey}
If $\eta$ is an affine generic character of $I_H^+$, then $\pi_H$ is simple supercuspidal.
\end{lem}

\begin{proof}
By the Frobenius reciprocity, we have
\[
\Hom_{H}\Bigl(\cInd_{I_H^+}^{H}\eta, \pi_H\Bigr) \cong \Hom_{I_H^+}\Bigl(\eta, \pi_H|_{I_H^+}\Bigr)\neq 0.
\]
By Proposition \ref{prop:ssc} and Remark \ref{rem:ssc}, the representation $\cInd_{I_H^+}^{H}\eta$ is a direct sum of two simple supercuspidal representations.
Since $\pi_H$ is irreducible, it is equivalent to one of them.
%
\end{proof}

From this lemma, we are reduced to proving the affine genericity of $\eta$.
We first prove two technical lemmas which will be needed in the proof of the affine genericity of $\eta$.

\begin{lem}\label{lem:hyperspecial}
Let $H_0:=\SO_{2n+1}(\mcO)$ be a hyperspecial subgroup of $H$.
Let $y\in H$.
If $y$ satisfies $yhy^{-1} \in H_0$ for an affine generic element $h \in H$, then $y$ belongs to $H_0 \sqcup H_0\varphi'_{1}$.
\end{lem}

\begin{proof}
Let $y \in H$ satisfying $yhy^{-1} \in H_0$ for an affine generic element $h \in I_H^+$.
As affine genericity is preserved by $I_H^+$-conjugation, any element of $H_0yI_H^+$ satisfies the same condition as $y$.
It suffices to show $H_0yI_H^+ \subseteq H_0\sqcup H_0\varphi'_{1}$.

The isomorphism in Proposition \ref{prop:I^+W} induces an isomorphism
\[
H_0\backslash H/I_H^+ \cong \bigl(\bigl(N_{T}\cap H_0\bigr)/(T_{H})_1\bigr)\big\backslash\bigl(N_{T}/(T_{H})_1\bigr)
\]
(see \cite[Proposition 8]{MR2435422}) 
and the set
\[
T_{H}\bigl(\varpi^\Z\bigr):=\bigl\{\diag(\varpi^{r_1}, \ldots, \varpi^{r_n}, 1, \varpi^{-r_n}, \ldots, \varpi^{-r_1}) \mid r_1, \ldots, r_n\in\Z \bigr\}
\]
contains a set of representatives of the right-hand side. 
Hence we may assume 
\begin{align*}
y&=\diag(t_1, \ldots, t_{2n+1})\\
&=\diag(\varpi^{r_1}, \ldots, \varpi^{r_n}, 1, \varpi^{-r_n}, \ldots, \varpi^{-r_1}) \in T_{H}\bigl(\varpi^\Z\bigr).
\end{align*}
Since $(yhy^{-1})_{ij}=t_i/t_j\cdot h_{ij}$ and $yhy^{-1}\in H_{0}$, we have the following inequalities:
\begin{align*}
r_1-r_2 &\geq0, \\
&\vdotswithin{=}\\
r_{n-1}-r_n &\geq0,\\
r_n &\geq0 \text{, and}\\
-r_1-r_2 &\geq-1.
\end{align*}
Therefore $(r_1, \ldots, r_n)$ is either $(0, \ldots, 0)$ or $(1, 0, \ldots, 0)$.

In the former case, we have $H_0yI_H^+=H_0$.

In the latter case, we have 
\[
H_0yI_H^+ = H_0\varphi'_{1}I_H^+ = H_0I_H^+\varphi'_{1} =H_0\varphi'_{1}
\]
(recall that $\varphi'_{1}$ normalizes $I_H^+$).
This completes the proof.
\end{proof}

\begin{lem}\label{lem:maxparah}
Let $\beta\in\Pi_{\mathbf{H}}$ be a simple affine root of $\mathbf{T}_{\mathbf{H}}$ in $\mathbf{H}$.
Then the subgroup $\lan I_H^{++}, U_\beta\ran$ of $H$ contains $H_\bold{x}^+$ for some point $\bold{x}$ of the apartment $\mathcal{A}(\mathbf{H}, \mathbf{T}_\mathbf{H})$, where $H_\bold{x}^+$ is the pro-unipotent radical of the parahoric subgroup of $H$ associated to $\bold{x}$.
\end{lem}

\begin{proof}
We recall that the pro-unipotent radical of the parahoric subgroup associated to a point $\bold{x}\in\mathcal{A}(\mathbf{H}, \mathbf{T}_{\mathbf{H}})$ and $I_{H}^{++}$ are given by
\begin{align*}
H_\bold{x}^+ &= \lan T_1, U_\alpha \mid \alpha\in\Psi_{\mathbf{H}}, \,\alpha(\bold{x})>0\ran, \text{ and}\\
I_{H}^{++} &= \lan T_1, U_\alpha \mid \alpha\in\Psi_{\mathbf{H}}^{+}\setminus\Pi_{\mathbf{H}}\ran,
\end{align*}
respectively.
Therefore we have to find a point $\bold{x}$ satisfying the following condition:
\[
\text{If an affine root $\alpha$ satisfies $\alpha(\bold{x})>0$, then we have $\alpha\in(\Psi_{\mathbf{H}}^{+}\setminus\Pi_{\mathbf{H}})\cup\{\beta\}$.} 
\tag{$\ast$}
\]

We first consider the case where $\beta=-e_1-e_2+1$.
In this case, we take $\bold{x}=0$.
If $\alpha=a+r\in\Psi_{\mathbf{H}}$ satisfies $\alpha(\bold{x})>0$, then we have $r>0$.
Hence $\alpha\in\Psi_{\mathbf{H}}^{+}\setminus\Pi_{\mathbf{H}}$ or $\alpha=\beta$.

We next consider the case where $\beta=e_{1}-e_{2}$.
However, since $\varphi'_{1}U_{e_{1}-e_{2}}\varphi'^{-1}_{1}=U_{-e_{1}-e_{2}+1}$, this case is reduced to the previous case.

We finally consider the other cases.
We define a point $\bold{x}\in\mathcal{A}(\mathbf{H}, \mathbf{T}_\mathbf{H})$ as follows:
\[
\bold{x} := 
\begin{cases}
 (\check{e}_1+\cdots+\check{e}_i)/2 & (\beta=e_i-e_{i+1} \text{ for } 1<i<n),\\
 (\check{e}_1+\cdots+\check{e}_n)/2 & (\beta=e_n).
\end{cases}
\]
It is enough to prove $(\ast)$ for $\alpha=a+r_{a}\in\Psi_{\mathbf{H}}$, where $r_a\in\Z$ is the smallest integer satisfying $a(\bold{x})+r_a>0$.

\begin{itemize}
\item
If $a \in \Phi_{\mathbf{H}}^{+}\setminus\Delta_{\mathbf{H}}$, then we have $0\leq a(\bold{x})\leq1$.
Hence $\alpha(\bold{x})>0$ implies that $r_a\geq0$.
Therefore $\alpha=a+r_a$ belongs to $\Psi_{\mathbf{H}}^{+}\setminus\Pi_{\mathbf{H}}$.

\item
If $a \in\Delta_{\mathbf{H}}$, then we have 
\[
a(\bold{x})=
\begin{cases}
 0 & (a\neq\beta),\\
 1/2 & (a=\beta).
\end{cases}
\]
Hence we have
\[
\alpha(\bold{x})>0 \Longleftrightarrow 
\begin{cases}
 r_a>0 & (a\neq\beta),\\
 r_a\geq0 & (a=\beta) .
\end{cases}
\]
Therefore we have either $\alpha\in\Psi_{\mathbf{H}}^{+}\setminus\Pi_{\mathbf{H}}$ or $\alpha=\beta$.

\item
If $a \in \Phi_{\mathbf{H}}^{-}\setminus\{-e_{1}-e_{2}\}$, then we have $-1\leq a(\bold{x})\leq0$.
Hence $\alpha(\bold{x})>0$ implies that $r_a\geq1$.
Therefore $\alpha=a+r_a$ belongs to $\Psi_{\mathbf{H}}^{+}\setminus\Pi_{\mathbf{H}}$.

\item
If $a=-e_{1}-e_{2}$, then we have $\alpha(\bold{x})=r_a-1$.
Hence $\alpha(\bold{x})>0$ implies that $r_a\geq2$, and we have $\alpha\in\Psi_{\mathbf{H}}^{+}\setminus\Pi_{\mathbf{H}}$.

\end{itemize}
\end{proof}

\begin{prop}\label{prop:simplity}
Every character $\eta$ of $I_H^+$ which is contained in $\pi_H^{I_H^{++}}$ is affine generic.
In particular, the representation $\pi_H$ is simple supercuspidal.
\end{prop}

\begin{proof}
We suppose that $\eta$ is not affine generic.

Since $\eta$ is not affine generic, there exists a simple affine root $\beta \in \Pi_{\mathbf{H}}$ such that its affine root subgroup $U_\beta$ is contained in $\mathrm{Ker}(\eta)$.
Hence $\pi_H$ has a nonzero $\lan I_H^{++}, U_\beta\ran$-fixed vector.
Then, by Lemma \ref{lem:maxparah}, $\pi_{H}$ also has a nonzero $H_\bold{x}^+$-fixed vector for some point $\bold{x}\in\mathcal{A}(\mathbf{H}, \mathbf{T}_\mathbf{H})$, where $H_\bold{x}^+$ is the pro-unipotent radical of the parahoric subgroup associated to $\bold{x}$.
Therefore the depth of $\pi_H$ is zero.

On the other hand, $\pi_H$ is generic with respect to any Whittaker datum.
Indeed, by Proposition 8.3.2 in \cite{MR3135650}, the $L$-packet $\Pi_{\phi_{H}}$ of $\phi_H$ contains a generic representation.
However $\Pi_{\phi_{H}}$ is a singleton consisting of $\pi_{H}$.
Thus $\pi_{H}$ is generic.

Then, by Lemma 6.1.2 in \cite{MR2480618}, the generic depth-zero supercuspidal representation $\pi_H$ can be obtained by the compact induction of a representation of the reduction of a hyperspecial subgroup of $H$.
Since $\mathbf{H}=\SO_{2n+1}$ is an adjoint group, all hyperspecial subgroups of $H$ are conjugate (see 2.5 in \cite{MR546588}).
Thus we may assume that the hyperspecial subgroup is $H_0=\SO_{2n+1}(\mcO)$.

Let $H_0^+$ be the pro-unipotent radical of $H_0$.
Let $\pi_H \cong \cInd_{H_0}^{H} \dot{\rho}$, where $\dot{\rho}$ is the inflation of a representation $\rho$ of $H_0/H_0^+\cong\SO_{2n+1}(k)$ to $H_0$.
Then, by the character formula (Theorem \ref{thm:CF}), we have
\[
\Theta_{\pi_H}\bigl(\mathfrak{N}(1+\varphi_u)\bigr) = \sum_{\begin{subarray}{c} y\in H_0 \backslash H\\ y\mathfrak{N}(1+\varphi_u)y^{-1} \in H_0 \end{subarray}} \tr \bigl(\dot{\rho}\bigl(y\mathfrak{N}(1+\varphi_u)y^{-1}\bigr)\bigr)
\]
for any $u \in k^{\times}$.
By Lemma \ref{lem:hyperspecial}, we have
\[
\Theta_{\pi_H}\bigl(\mathfrak{N}(1+\varphi_u)\bigr)=\tr \bigl(\dot{\rho}\bigl(\mathfrak{N}(1+\varphi_u)\bigr)\bigr)+\tr \bigl(\dot{\rho}\bigl(\varphi'_{1}\mathfrak{N}(1+\varphi_u)\varphi'^{-1}_{1}\bigr)\bigr).
\]
As 
\[
t\varphi'_{1}\mathfrak{N}(1+\varphi_u)\varphi'^{-1}_{1}t^{-1} = \mathfrak{N}(1+\varphi_u),
\]
where $t=\diag(2/u, 1, \ldots, 1, u/2) \in H_{0}$, 
we have
\[
\Theta_{\pi_H}\bigl(\mathfrak{N}(1+\varphi_u)\bigr)=2\tr\bigl(\dot{\rho}\bigl(\mathfrak{N}(1+\varphi_u)\bigr)\bigr).
\]
The image of $\mathfrak{N}(1+\varphi_u)$ in $H_0/H_0^+$ is independent of $u$, hence we can conclude that $\Theta_{\pi_H}(\mathfrak{N}(1+\varphi_u))$ is independent of $u$.
However we have
\[
\Theta_{\pi_H}\bigl(\mathfrak{N}(1+\varphi_u)\bigr)=\Kl_{2^{2(n-1)}u}^{n+1}(\psi; 1, 2, \ldots, 2, 1)
\]
by Corollary \ref{cor:CR}, 
and this is not constant on $u$ by Corollary \ref{cor:Kl}.
This is a contradiction.
\end{proof}

\subsection{Main theorem}
\begin{prop}\label{prop:main}
Let $a \in k^{\times}, \zeta \in \{\pm1\}$.
Then the representation $\pi_H$ associated to $\pi_{a,\zeta}$ is given by $\pi'_{a/2, \zeta}$.
\end{prop}

\begin{proof}
By replacing the fixed uniformizer $\varpi$, we may assume that $a=1$.
By Proposition \ref{prop:simplity}, $\pi_H$ is simple supercuspidal.
Let $\pi_H$ be $\pi'_{b, \xi}$, where $b \in k^{\times}$ and $\xi \in \{\pm1\}$.
Our task is to determine $b$ and $\xi$.

First, we consider $b$.
For $u \in k^{\times}$, we have
\[
\Theta_{\pi_H}\bigl(\mathfrak{N}(1+\varphi_u)\bigr)=\Theta_{\pi_{1,\zeta}, \theta}(1+\varphi_u)=\Kl_{2^{2(n-1)}u}^{n+1}(\psi; 1, 2, \ldots, 2, 1)
\]
by Corollary \ref{cor:CR} (1).
On the other hand, since $\mathfrak{N}(1+\varphi_u)\in I_H^+$ is an affine generic element with its simple affine components $(2, \ldots, 2, u)$, we have
\[
\Theta_{\pi_H}\bigl(\mathfrak{N}(1+\varphi_u)\bigr) = \Kl_{2^{2(n-1)+1}bu}^{n+1}(\psi; 1, 2, \ldots, 2, 1)
\]
by Proposition \ref{prop:charSO}.
Hence $b=1/2$ by Proposition \ref{prop:Fourier}.

Next, we consider $\xi$.
For $u \in k^{\times}$, we have
\[
\Theta_{\pi_H}\bigl(\varphi'_{u^2/2}\mathfrak{N}(1+\varphi_{u^2})\bigr)=\Theta_{\pi_{1,\zeta}, \theta}\bigl(\varphi_{u^2}(1+\varphi_{u^2})\bigr)=\zeta\left(\Kl_{2^nu}^n(\psi)+\Kl_{-2^nu}^n(\psi)\right)
\]
by Corollary \ref{cor:CR} (2).
On the other hand, $\varphi'_{u^2/2}\mathfrak{N}(1+\varphi_{u^2})$ is strongly regular semisimple (by Proposition \ref{prop:N} (2)) and  the simple affine components of $(\varphi'_{u^2/2}\mathfrak{N}(1+\varphi_{u^2}))^2 \in I_H^+$ are given by $(4, \ldots, 4, 2u^2)$.
Hence $\varphi'_{u^2/2}\mathfrak{N}(1+\varphi_{u^2})$ satisfies the assumption of Proposition $\ref{prop:charSO'}$, and we have
\[
\Theta_{\pi_H}\bigl(\varphi'_{u^2/2}\mathfrak{N}(1+\varphi_{u^2})\bigr) = \xi\left(\Kl_{2^nu}^n(\psi)+\Kl_{-2^nu}^n(\psi)\right).
\]
By summing up over $u \in k^{\times}$, we have 
\[
2\zeta\sum_{u \in k^{\times}} \Kl_u^n(\psi)=2\xi\sum_{u \in k^{\times}} \Kl_u^n(\psi).
\]
Since $\sum_{u \in k^{\times}} \Kl_u^n(\psi)=(-1)^n\neq0$ by Proposition \ref{prop:Kl}, we have $\xi=\zeta$.
\end{proof}

In summary, we get the following result.

\begin{thm}[Main theorem]\label{thm:main}
Let $b \in k^{\times}$ and $\xi \in \{\pm1\}$.
Under the parametrizations as in Sections $2.3$ and $2.4$, we have the following:
\begin{enumerate}
\item
The $L$-packet containing the simple supercuspidal representation $\pi'_{b, \xi}$ of $\SO_{2n+1}(F)$ is a singleton.
In particular, the character of $\pi'_{b, \xi}$ is stable.
\item
The lifting of the simple supercuspidal representation $\pi'_{b, \xi}$ of $\SO_{2n+1}(F)$ to $\GL_{2n}(F)$ is again simple supercuspidal, and given by $\pi_{2b, \xi}$.
\end{enumerate}
\end{thm}

\begin{rem}
From this result, we know that the $L$-parameter of $\pi'_{b, \xi}$ is equal to that of $\pi_{2b, \xi}$.
On the other hand, $L$-parameters of simple supercuspidal representations of general linear groups have been determined explicitly by the works of \cite{MR2148193} and \cite{Imai:2015aa}.
Therefore we can get an explicit description of $L$-parameters of simple supercuspidal representations of $\SO_{2n+1}(F)$. 
\end{rem}

\appendix
\section{Gauss and Kloosterman sums}\label{sec:Kl}

\begin{defn}[Gauss sum]
Let $\chi \colon \F_q^{\times} \ra \C^{\times}$ be a multiplicative character and $\psi \colon \F_q \ra \C^{\times}$ a nontrivial additive character.
We define the Gauss sum with respect to the pair $(\chi, \psi)$ by
\[
G(\chi, \psi) := \sum_{t \in \F_q^{\times}} \chi(t)\psi(t).
\]
\end{defn}

The following properties are well-known (see e.g.\ \cite[page 501]{MR0029393}).

\begin{prop}\label{prop:Gauss}
\begin{enumerate}
 \item If $\chi$ is trivial, then $G(\chi, \psi)=-1$.
 \item If $\chi$ is nontrivial, then $|G(\chi, \psi)|=\sqrt{q}$.
\end{enumerate}
In particular, the sum $G(\chi, \psi)$ is nonzero.
\end{prop}

\begin{defn}[Kloosterman sum, \cite{MR955052}]\label{defn:Kl}
Let  $\psi \colon \F_q \ra \C^{\times}$ be a nontrivial additive character, $a$ an element of $\F_q$, and $b_1, \ldots, b_n$ positive integers.
We define the Kloosterman sum with respect to $(\psi, a, b_1, \ldots, b_n)$ by 
\[
\Kl_a^n(\psi; b_1, \ldots, b_n) := \sum_{\begin{subarray}{c} x_1, \ldots , x_n \in \F_q\\ x_1^{b_1} \cdots x_n^{b_n} = a \end{subarray}} \psi(x_1 + \cdots + x_n).
\] 
When $b_1 = \cdots = b_n =1$, we denote it simply by $\Kl_a^n(\psi)$.
\end{defn}

\begin{prop}\label{prop:Kl}
For a nontrivial additive character $\psi$, positive integers $(b_1, \ldots, b_n)$, and a multiplicative character $\chi \colon \F_q^{\times} \ra \C^{\times}$, the following properties hold:
\begin{enumerate}
 \item 
 \[
 \Kl_0^n(\psi; b_1, \ldots, b_n) = (-1)^{n-1},
 \]
 \item 
 \[
 \sum_{a \in \F_q} \Kl_a^n(\psi; b_1, \ldots, b_n) = 0 \text{, and}
 \]
 \item 
 \[
 \sum_{a \in \F_q^{\times}} \chi(a) \Kl_a^n(\psi; b_1, \ldots, b_n) = \prod_{i=1}^n G(\chi^{b_i}, \psi).
 \]
\end{enumerate}
\end{prop}

\begin{proof}
\begin{enumerate}
 \item By the inclusion-exclusion principle, 
 \begin{align*}
 \Kl_0^n(\psi; b_1, \ldots, b_n)
 &= \sum_{\begin{subarray}{c}x_1, \ldots , x_n \in \F_q\\ x_1^{b_1} \cdots x_n^{b_n} = 0 \end{subarray}} \psi(x_1 + \cdots + x_n)\\
 &= \sum_{\begin{subarray}{c}I \subseteq \{1, \ldots, n\}\\ I\neq \emptyset \end{subarray}} \sum_{\begin{subarray}{c}x_1, \ldots , x_n \in \F_q\\ x_i = 0, i \in I\end{subarray}} (-1)^{|I|-1}\psi(x_1 + \cdots + x_n) \\
 &= (-1)^{n-1}.
 \end{align*} 
 \item
 We have 
 \begin{align*}
 \sum_{a \in \F_q} \Kl_a^n(\psi; b_1, \ldots, b_n) 
 &= \sum_{x_1, \ldots , x_n \in \F_q} \psi(x_1 + \cdots + x_n) \\
 &= \Biggl( \sum_{t \in \F_q} \psi(t) \Biggr)^n \\
 &=0.
 \end{align*}
 \item
 We have 
 \begin{align*}
 \sum_{a \in \F_q^{\times}} \chi(a) \Kl_a^n(\psi; b_1, \ldots, b_n)
 &= \sum_{x_1, \ldots , x_n \in \F_q^{\times}} \chi(x_1^{b_1} \cdots x_n^{b_n}) \psi(x_1 + \cdots + x_n)\\
 &= \prod_{i=1}^n \Biggl( \sum_{x_i \in \F_q^{\times}} \chi(x_i^{b_i}) \psi(x_i)\Biggr)\\
 &= \prod_{i=1}^n G(\chi^{b_i}, \psi).
 \end{align*}
\end{enumerate}
\end{proof}

\begin{cor}\label{cor:Kl}
\begin{enumerate}
 \item The sum $\Kl_a^n(\psi; b_1, \ldots, b_n)$ is not constant on $a \in \F_q^{\times}$.
 \item For some $a \in \F_q^{\times}$, $\Kl_a^n(\psi; b_1, \ldots, b_n) \neq 0$． 
\end{enumerate}
\end{cor}

\begin{proof}
Assume that the sum $\Kl_a^n(\psi; b_1, \ldots, b_n)$ is constant on $a \in \F_q^{\times}$.
If we take a nontrivial multiplicative character $\chi$, then by Proposition \ref{prop:Kl} (3) we have 
\[
\prod_{i=1}^n G(\chi^{b_i}, \psi)=0.
\]
However this contradicts Proposition \ref{prop:Gauss}.
Thus we get (1), and (2) immediately follows from (1).
\end{proof}

The following proposition is a slight generalization of {\cite[Lemma 3.4]{MR3843393}}:
\begin{prop}\label{prop:Fourier}
Let $a, b \in \F_q^{\times}$.
If \[
\Kl_{ta}^n(\psi; b_1, \ldots, b_n) = \Kl_{tb}^n(\psi; b_1, \ldots, b_n)
\] for every $t \in \F_q^{\times}$, 
then we have $a=b$.
\end{prop}

\begin{proof}
%
We may assume $b=1$.
It suffices to show that if $a \neq 1$, then there exists $t \in \F_q^{\times}$ such that $\Kl_{ta}^n(\psi; b_1, \ldots, b_n) \neq \Kl_{t}^n(\psi; b_1, \ldots, b_n)$.
Let us assume $a\neq1$, then we can take a multiplicative character $\chi$ satisfying $\chi(a) \neq 1$.
Then we have 
\begin{align*}
&\sum_{t \in \F_q^{\times}} \chi(t) \bigl( \Kl_{ta}^n(\psi; b_1, \ldots, b_n) - \Kl_{t}^n(\psi; b_1, \ldots, b_n) \bigr) \\
&= \bigl(\chi(a)^{-1} - 1\bigr) \sum_{t \in \F_q^{\times}} \chi(t) \Kl_t^n(\psi; b_1, \ldots, b_n) \\
&= \bigl(\chi(a)^{-1} - 1\bigr) \prod_{i=1}^n G(\chi^{b_i}, \psi) \\
&\neq 0.
\end{align*}
Hence $\Kl_{ta}^n(\psi; b_1, \ldots, b_n) \neq \Kl_{t}^n(\psi; b_1, \ldots, b_n)$ for some $t \in \F_q^{\times}$.
\end{proof}
%
%
%


\end{document}